\pdfoutput=1
\RequirePackage{ifpdf}
\ifpdf 
\documentclass[pdftex]{sigma}
\else
\documentclass{sigma}
\fi

\usepackage[all]{xy}

\numberwithin{equation}{section}

\newtheorem{Theorem}{Theorem}[section]
\newtheorem*{Theorem*}{Theorem}
\newtheorem{Corollary}[Theorem]{Corollary}
\newtheorem{Lemma}[Theorem]{Lemma}
\newtheorem{Proposition}[Theorem]{Proposition}
 { \theoremstyle{definition}

\newtheorem{Remark}[Theorem]{Remark} }

\begin{document}

\newcommand{\arXivNumber}{2106.06764}

\renewcommand{\PaperNumber}{010}

\FirstPageHeading

\ShortArticleName{Relationships Between Hyperelliptic Functions of Genus 2 and Elliptic Functions}

\ArticleName{Relationships Between Hyperelliptic Functions\\ of Genus 2 and Elliptic Functions}

\Author{Takanori AYANO~$^{\rm a}$ and Victor M.~BUCHSTABER~$^{\rm b}$}

\AuthorNameForHeading{T.~Ayano and V.M.~Buchstaber}

\Address{$^{\rm a)}$~Osaka City University, Advanced Mathematical Institute, \\
\hphantom{$^{\rm a)}$}~3-3-138 Sugimoto, Sumiyoshi-ku, Osaka, 558-8585, Japan}
\EmailD{\href{mailto:ayano@sci.osaka-cu.ac.jp}{ayano@sci.osaka-cu.ac.jp}}
\URLaddressD{\url{https://researchmap.jp/ayano75?lang=en}}

\Address{$^{\rm b)}$~Steklov Mathematical Institute of Russian Academy of Sciences,\\
\hphantom{$^{\rm b)}$}~8 Gubkina Street, Moscow, 119991, Russia}
\EmailD{\href{mailto:buchstab@mi-ras.ru}{buchstab@mi-ras.ru}}

\ArticleDates{Received June 15, 2021, in final form January 20, 2022; Published online February 01, 2022}

\Abstract{The article is devoted to the classical problems about the relationships bet\-ween elliptic functions and hyperelliptic functions of genus~2. It contains new results, as well as a~derivation from them of well-known results on these issues. Our research was moti\-vated by applications to the theory of equations and dynamical systems integrable in hyperelliptic functions of genus~2. We~consider a hyperelliptic curve~$V$ of genus~2 which admits a~morphism of degree 2 to an elliptic curve. Then there exist two elliptic curves $E_i$, $i=1,2$, and morphisms of degree 2 from $V$ to $E_i$. We~construct hyperelliptic functions associated with~$V$ from the Weierstrass elliptic functions associated with~$E_i$ and describe them in terms of the fundamental hyperelliptic functions defined by the logarithmic derivatives of the two-dimensional sigma functions. We~show that the restrictions of hyperelliptic functions associated with $V$ to the appropriate subspaces in $\mathbb{C}^2$ are elliptic functions and describe them in terms of the Weierstrass elliptic functions associated with $E_i$. Further, we express the hyperelliptic functions associated with $V$ on $\mathbb{C}^2$ in terms of the Weierstrass elliptic functions associated with $E_i$. We~derive these results by describing the homomorphisms between the Jacobian varieties of the curves $V$ and $E_i$ induced by the morphisms from $V$ to $E_i$ explicitly.}


\Classification{14H40; 14H42; 14K25; 32A20; 33E05}

\begin{flushright}
\begin{minipage}{73mm}
\it Dedicated to the memory of the remarkable\\ mathematician Victor Enolski $($1945--2019$)$
\end{minipage}
\end{flushright}

\section{Introduction}

Victor Enolski's contributions to mathematics, mathematical and theoretical physics are reflected in the memorial survey \cite{Pre}.
Our article is devoted to the issues that were in the center of his attention; it contains our new results on these issues
and a description of their connection with classical and relatively recent results, including the results obtained by Victor Enolski together with his co-authors.

The problem whether the Jacobian variety of a hyperelliptic curve of genus 2 is isogenous to the direct product of elliptic curves is considered in many papers of number theory
(e.g., \cite{BCMS2021, B-L-S-2013,BD, Cassels, F2, FK, Ku, P,J.P.Serre-2020, S, W}).
The problems discussed in our article are related to the well-known open problem about rational points on the Jacobian variety of a curve of genus 2 \cite{Howe-Leprevost-Poonen-Large-torsion-subgroups-2000,P}.
The problem on the isogeny between the Jacobian variety of a curve of genus 2 and the direct product of elliptic curves is naturally connected with the following well-known problem (to which our research was devoted): Let solutions of differential equations and dynamical systems in terms of hyperelliptic functions of genus 2 are given.
Under conditions when a reduction of these functions to elliptic functions is possible, find an explicit form of these solutions in terms of elliptic functions.
In this paper, we consider the algebraic curves over $\mathbb{C}$.
The reductions of the Riemann theta functions and the hyperelliptic functions turned out to be important in real physical problems (cf.\ \cite{FCoppini2020,AOSmirnov}).
In~\cite{FCoppini2020,AOSmirnov}, a hyperelliptic curve of genus 2 defined by $y^2=f(x)$ with a square-free polynomial $f(x)$ of degree 6 is considered.
In the case when the curve is defined by $y^2=\widetilde{f}(x)$ with a square-free polynomial $\widetilde{f}(x)$ of degree 5, its sigma function is determined by the coefficients of the polynomial $\widetilde{f}(x)$.
In the case when the curve is defined by $y^2=f(x)$ with a square-free polynomial $f(x)$ of degree~6, one of the branch points can be chosen and an isomorphism can be explicitly set with a curve defined by $y^2=\widetilde{f}(x)$ with a square-free polynomial $\widetilde{f}(x)$ of degree 5. Thus, in the case under discussion, the sigma function and the corresponding hyperelliptic functions are determined by the coefficients of the polynomial of degree 6 and the choice of one of its zeros. In some cases, an explicit transformation uses several branch points (cf.\ Proposition~\ref{2021.11.22.3}).

For a positive integer $\Delta$, let $\mathcal{H}_{\Delta}$ be the set of complex $2\times2$ matrix $\tau=\left(\begin{smallmatrix}\tau_{11}&\tau_{12}\\\tau_{12}&\tau_{22}\end{smallmatrix}\right)$ such that ${}^t\tau=\tau$, $\operatorname{Im} \tau>0$, and
\begin{gather*}
h_1\tau_{11}+h_2\tau_{12}+h_3\tau_{22}+h_4\big(\tau_{12}^2-\tau_{11}\tau_{22}\big)+h_5=0
\end{gather*}
for some integers $h_1$, $h_2$, $h_3$, $h_4$, $h_5$ satisfying $\Delta=h_2^2-4(h_1h_3+h_4h_5)$.
Then $\mathcal{H}_{\Delta}$ is called \textit{Humbert surface} with invariant~$\Delta$ \cite{Humbert-1899,HumbertG-1900,HumbertG-1901}.

\begin{Theorem}[{\cite{K}, \cite[Theorem 5.10]{BE2}, \cite[Theorem 7.4]{BEL}}]
Let $H$ be a hyperelliptic curve of genus~$2$ and $N$ be an integer such that $N\ge2$.
There exist an elliptic curve $W$ and a morphism $H\to W$ of degree $N$ if and only if a normalized period matrix $\tau=\left(\begin{smallmatrix}\tau_{11}&\tau_{12}\\\tau_{12}&\tau_{22}\end{smallmatrix}\right)$ of $H$
belongs to the Humbert surface $\mathcal{H}_{\Delta}$ with $\Delta=N^2$.
\end{Theorem}

In~\cite{B-W-2003}, the geometric characterizations of the Humbert surfaces in terms of the presence of certain curves on the associated Kummer plane are given.
In~\cite[Section~3.1]{BE}, \cite[Section~5.3.3]{BE2}, \cite[Section~7.3.3]{BEL}, \cite[Section~6.3, Section~11.5]{BEL-2012}, \cite{Enolskii-Salerno-1996}, the Humbert surfaces are considered and relationships between the hyperelliptic functions associated with a curve of genus 2 and the Jacobi elliptic functions are derived.
In~\cite[p.~3448]{BE}, \cite[p.~366]{BE2}, \cite[pp.~79,~80]{BEL}, \cite[pp.~99,~175]{BEL-2012},~\cite{Enolskii-Salerno-1996}, it is mentioned that the Jacobi elliptic functions and the hyperelliptic functions associated with a curve of genus 2 give coordinates on the Kummer surfaces.
In~\cite{BCMS2021}, algebraic correspondences between Kummer surfaces associated with the Jacobian variety of a~hyperelliptic curve of genus~2 admitting a morphism of degree~2 to an elliptic curve and
Kummer surfaces associated with the product of two non-isogenous elliptic curves are considered.
In the equations~(2.56) and~(2.57) of~\cite{BCMS2021}, the formulae in \cite[p.~3448]{BE}, \cite[pp.~99,~100]{BEL-2012}, \cite{Enolskii-Salerno-1996} are described in terms of coordinates on the Kummer surfaces.
In this paper, by an approach different from \cite{BE, BE2, BEL, BEL-2012, Enolskii-Salerno-1996}, we derive relationships between the hyperelliptic functions associated with the curve of genus~2 and the Weierstrass elliptic functions.

For $\alpha, \beta\in\mathbb{C}$ satisfying $\alpha^2,\beta^2\neq0,1$, $\alpha^2\neq\beta^2$, and $\alpha^2\beta^2\neq1$,
we consider the nonsingular hyperelliptic curve of genus $2$
\begin{gather*}
V=\big\{(x,y)\in\mathbb{C}^2\mid y^2=x(x-1)\big(x-\alpha^2\big)\big(x-\beta^2\big)\big(x-\alpha^2\beta^2\big)\big\}.
\end{gather*}
For a hyperelliptic curve $H$ of genus~2, it is known that there exist an elliptic curve $W$ and a~morphism $H\to W$ of degree $2$ if and only if $H$ is isomorphic to the curve~$V$ for some $\alpha$,~$\beta$ (see Section~\ref{2021.11.22.1111}).
Let $E_1$ and $E_2$ be the elliptic curves defined by
\begin{gather*}
E_1=\bigg\{(X,Y)\in\mathbb{C}^2\,\bigg|\,Y^2
=X(X-1)\bigg(X-\frac{(\alpha-\beta)^2}{(\alpha\beta-1)^2}\bigg)\bigg\},
\\
E_2=\bigg\{(X,Y)\in\mathbb{C}^2\,\bigg|\,Y^2
=X(X-1)\bigg(X-\frac{(\alpha+\beta)^2}{(\alpha\beta+1)^2}\bigg)\bigg\}.
\end{gather*}
Then we can define the morphisms of degree 2 (see Section~\ref{2021.11.22.1111})
\begin{gather*}
\varphi_1\colon\quad V\to E_1,\qquad
(x,y)\mapsto(X,Y)=\bigg(\frac{(\alpha-\beta)^2x}{(x-\alpha\beta)^2},
\frac{(\alpha-\beta)^2y}{(\alpha\beta-1)(x-\alpha\beta)^3}\bigg),
\\
\varphi_2\colon\quad V\to E_2,\qquad
(x,y)\mapsto(X,Y)=\bigg(\frac{(\alpha+\beta)^2x}{(x+\alpha\beta)^2},
\frac{-(\alpha+\beta)^2y}{(\alpha\beta+1)(x+\alpha\beta)^3}\bigg).
\end{gather*}

Let $\operatorname{Jac}(V)$, $\operatorname{Jac}(E_1)$, and $\operatorname{Jac}(E_2)$ be the Jacobian varieties of $V$, $E_1$, and $E_2$, respectively.
The maps $\varphi_i$, $i=1,2$, induce the homomorphisms of the Jacobian varieties
\begin{gather*}
\psi_{i,*}\colon\ \operatorname{Jac}(V)\to\operatorname{Jac}(E_i),\qquad
\psi_i^*\colon\ \operatorname{Jac}(E_i)\to\operatorname{Jac}(V),\qquad
i=1,2.
\end{gather*}
In this paper, we describe the maps $\psi_{i,*}$ and $\psi_i^*$ explicitly in Propositions~\ref{3.6.1} and~\ref{2.15.1}.
Let $\sigma({\bf u})$, ${\bf u}={}^t(u_1,u_3)\in\mathbb{C}^2$, be the sigma function associated with $V$, which is the holomorphic function on $\mathbb{C}^2$, $\wp_{j,k}=-\partial_{u_j}\partial_{u_k}\log\sigma$, and $\wp_{j,k,\ell}=\partial_{u_{\ell}}\wp_{j,k}$, where $\partial_{u_j}=\frac{\partial}{\partial u_j}$. Then the functions $\wp_{j,k}$ and $\wp_{j,k,\ell}$ are the fundamental meromorphic functions on $\operatorname{Jac}(V)$.
For $i=1,2$, let $\wp_{E_i}$ be the Weierstrass elliptic function associated with the elliptic curve $E_i$.
Let $f_i({\bf u})=\wp_{E_i}(\psi_{i,*}({\bf u}))$ for ${\bf u}\in\operatorname{Jac}(V)$ and $i=1,2$.
Then the functions $f_1({\bf u})$ and $f_2({\bf u})$ are the meromorphic functions on $\operatorname{Jac}(V)$.
In this paper, we express $f_1({\bf u})$ and $f_2({\bf u})$ in terms of $\wp_{1,1}({\bf u}), \wp_{1,3}({\bf u})$, and $\wp_{3,3}({\bf u})$ explicitly in Theorem \ref{3.6.2} and Corollary \ref{3.12.111}. We~consider the homomorphisms of the Jacobian varieties
\begin{gather*}
\psi_i^*\colon\quad \operatorname{Jac}(E_i)\to\operatorname{Jac}(V),\qquad i=1,2.
\end{gather*}
For $v\in \operatorname{Jac}(E_i)$, the functions $\wp_{j,k}(\psi_i^*(v))$ and $\wp_{j,k,\ell}(\psi_i^*(v))$ are the meromorphic functions on~$\operatorname{Jac}(E_i)$.
In this paper, we describe the functions $\wp_{j,k}(\psi_i^*(v))$ and $\wp_{j,k,\ell}(\psi_i^*(v))$ in terms of $\wp_{E_i}(v)$ explicitly in Theorem~\ref{2.17.1} and Corollary~\ref{3.12.1111}, i.e.,
we show that the restrictions of the hyperelliptic functions $\wp_{j,k}$ and $\wp_{j,k,\ell}$ to the appropriate subspaces in $\mathbb{C}^2$ are elliptic functions and describe them in terms of the Weierstrass elliptic functions $\wp_{E_i}$.
Further, by using the addition formulae for $\wp_{j,k}$, which are given in \cite[Theorem 3.3]{G}, we express $\wp_{j,k}$ on $\mathbb{C}^2$ in terms of $\wp_{E_1}$ and $\wp_{E_2}$ explicitly in Theorem~\ref{3.25.1}.
From Theorem \ref{3.25.1}, we obtain the decompositions of the functions $\wp_{1,1}$, $\wp_{1,3}-\alpha^2\beta^2$, and $\wp_{3,3}$ into the products of meromorphic functions (Remark~\ref{5.8.111}),
a solution of the KdV-hierarchy in terms of functions constructed by the elliptic functions $\wp_{E_1}$ and $\wp_{E_2}$ (Remark \ref{3.28}), and
the expressions of the coordinates of the Kummer surface in terms of the Weierstrass elliptic functions $\wp_{E_1}$ and $\wp_{E_2}$ (Remark \ref{2021.11.22.2222}).
In Section~\ref{5.8.222}, we compare our results with the results of \cite{BE, BEL-2012, Enolskii-Salerno-1996}.

The elliptic sigma functions, which are defined and studied by Weierstrass, are generalized to the sigma functions associated with the hyperelliptic curves and many applications in integrable systems and mathematical physics are derived
(cf.~\cite{BEL-97-1, BEL-97-2, BEL, BEL-2012}).
In~\cite{AB2019}, we consider the solutions of the inversion problem of the ultra-elliptic integrals of the hyperelliptic curves of genus 2 in terms of the sigma functions associated with the hyperelliptic curves of genus 2.
In~\cite{AB2020}, we consider the series expansion of the sigma functions associated with the hyperelliptic curves of genus 2 by using the heat equations in a nonholonomic frame derived in \cite{BL-2004}.
The problem of the reduction of the hyperelliptic integrals to the elliptic integrals has been studied in many papers (e.g., \cite{BE0, BE, BE2, O-Bolza-1887, O-Bolza-1935}).
In particular, O. Bolza derived many examples of the reduction of hyperelliptic integrals to elliptic integrals (cf.~\cite{O-Bolza-1887, O-Bolza-1935}).
The problem of the reduction of a hyperelliptic integral to an elliptic integral is closely related to that of the morphism from a~hyperelliptic curve to an elliptic curve. In this paper, when a hyperelliptic curve of genus~2 admits a morphism of degree~2 to an elliptic curve, we derive the relationships between the hyperelliptic functions of genus~2, which are defined by the logarithmic derivatives of the sigma functions associated with the hyperelliptic curves of genus~2, and the Weierstrass elliptic functions, which are defined by the logarithmic derivatives of the elliptic sigma functions.

\section{Preliminaries}

In this section we recall the definitions of the Jacobi elliptic functions and the sigma functions for the curves of genus 1 and 2 and give facts about them which will be used later on. For the details of the Jacobi elliptic functions, see~\cite{H,WW}.
For the details of the sigma functions, see~\cite{BEL-2012}.

\subsection{The Jacobi elliptic functions}\label{4.27.11}
We denote the imaginary unit by ${\bf i}$.
Let $\mathbb{H}=\{\tau\in\mathbb{C}\mid \operatorname{Im} \tau>0\}$. For $(z,\tau)\in\mathbb{C}\times \mathbb{H}$ and $\varepsilon_1, \varepsilon_2\in\mathbb{R}$, the theta function with characteristics ${}^t(\varepsilon_1, \varepsilon_2)$ is defined by
\begin{gather*}
\theta\begin{bmatrix}\varepsilon_1\\ \varepsilon_2\end{bmatrix}(z,\tau)=\sum_{n\in\mathbb{Z}}\exp\left(\pi{\bf i}(n+\varepsilon_1)^2\tau+2\pi{\bf i}(n+\varepsilon_1)(z+\varepsilon_2)\right).
\end{gather*}
Let
\begin{alignat*}{3}
& \theta_{00}(z,\tau)=\theta\begin{bmatrix}0\\0\end{bmatrix}(z,\tau),\qquad&&
\theta_{01}(z,\tau)=\theta\begin{bmatrix}0\\1/2\end{bmatrix}(z,\tau),&
\\
& \theta_{10}(z,\tau)=\theta\begin{bmatrix}1/2\\0\end{bmatrix}(z,\tau),\qquad&&
\theta_{11}(z,\tau)=\theta\begin{bmatrix}1/2\\1/2\end{bmatrix}(z,\tau).&
\end{alignat*}
The Jacobi elliptic functions associated with $\tau\in \mathbb{H}$ are defined by
\begin{alignat*}{3}
& \mbox{sn}(u,\tau)=-\frac{\theta_{00}(0,\tau)\theta_{11}(z,\tau)}{\theta_{10}(0,\tau)\theta_{01}(z,\tau)},
\qquad &&
\mbox{cn}(u,\tau)=\frac{\theta_{01}(0,\tau)\theta_{10}(z,\tau)}{\theta_{10}(0,\tau)\theta_{01}(z,\tau)},&
\\
& \mbox{dn}(u,\tau)=\frac{\theta_{01}(0,\tau)\theta_{00}(z,\tau)}{\theta_{00}(0,\tau)\theta_{01}(z,\tau)},
\qquad&&
z=\frac{u}{\pi\theta_{00}(0,\tau)^2}.&
\end{alignat*}
The modulus $m$ of the Jacobi elliptic functions associated with $\tau$ is defined by
\begin{gather*}
m=\frac{\theta_{10}(0,\tau)^2}{\theta_{00}(0,\tau)^2}.
\end{gather*}
Then we have $m^2\neq0,1$.
Let $\omega'_{\tau}=\pi\theta_{00}(0,\tau)^2/2$, $\omega''_{\tau}=\omega'_{\tau}\tau$, and $\Lambda_{\tau}=\{4\omega'_{\tau}m_1+2\omega''_{\tau}m_2\mid m_1,m_2\in\mathbb{Z}\}$.
Let us consider the elliptic curve defined by
\begin{gather*}
E_{\tau}=\big\{(x,y)\in\mathbb{C}^2\mid y^2=\big(1-x^2\big)\big(1-m^2x^2\big)\big\}.
\end{gather*}
It is well known that the map
\begin{gather*}
\mathbb{C}/\Lambda_{\tau}\to E_{\tau},\qquad
u\mapsto(\mbox{sn}(u,\tau), \mbox{sn}'(u,\tau))
\end{gather*}
is the isomorphism, where $\mbox{sn}'(u,\tau)=\frac{\rm d}{{\rm d}u}\mbox{sn}(u,\tau)$.

\subsection{The elliptic sigma function}\label{2021.11.20.1111}

We set
\begin{gather*}
M_1(x)=x^3+\lambda_2x^2+\lambda_4x+\lambda_6,\qquad
\lambda_i\in\mathbb{C}.
\end{gather*}
We assume that $M_1(x)$ has no multiple root. We~consider the elliptic curve
\begin{gather*}
E=\big\{(x,y)\in\mathbb{C}^2\mid y^2=M_1(x)\big\}.
\end{gather*}
For $(x,y)\in E$, let
\begin{gather*}
\omega=-\frac{{\rm d}x}{2y},
\end{gather*}
which is the basis of the vector space of holomorphic one forms on $E$.
Further, let
\begin{gather*}
\eta=-\frac{x}{2y}{\rm d}x,
\end{gather*}
which is the meromorphic one form on $E$ with a pole only at $\infty$.
Let $\{\mathfrak{a},\mathfrak{b}\}$ be a canonical basis in the one-dimensional homology group of the curve $E$. We~define the periods by
\begin{gather*}
2\omega'=\int_{\mathfrak{a}}\omega, \qquad
2\omega''=\int_{\mathfrak{b}}\omega,\qquad
-2\eta'=\int_{\mathfrak{a}}\eta, \qquad
-2\eta''=\int_{\mathfrak{b}}\eta.
\end{gather*}
The normalized period has the form $\tau=(\omega')^{-1}\omega''$.
The sigma function associated with $E$ is defined by
\begin{gather*}
\sigma(u)=\frac{2\omega'}{\theta_{11}'(0,\tau)}\exp\bigg(\frac{\eta'}{2\omega'}u^2\bigg)
\theta_{11}\bigg(\frac{u}{2\omega'},\tau\bigg),
\end{gather*}
where $\theta_{11}'(0,\tau)=\frac{\partial}{\partial z}\theta_{11}(z,\tau)|_{z=0}$.
For $m_1, m_2\in\mathbb{Z}$, let $\Omega=2\omega'm_1+2\omega''m_2$.
Then, for $u\in\mathbb{C}$, we have
\begin{gather*}
\sigma(u+\Omega)/\sigma(u)=(-1)^{m_1-m_2+m_1m_2}
\exp\big\{(2\eta'm_1+2\eta''m_2)(u+\omega'm_1+\omega''m_2)\big\}.
\end{gather*}
We set $\deg\lambda_{2i}=2i$ for $1\le i\le 3$.
The sigma function $\sigma(u)$ is an entire function on $\mathbb{C}$ and the series expansion of $\sigma(u)$ around $u=0$ has the following form:
\begin{gather*}
\sigma(u)=u+\sum_{n\ge3}\mu_nu^n,
\end{gather*}
where the coefficient $\mu_n$ is a homogeneous polynomial in $\mathbb{Q}[\lambda_2, \lambda_4, \lambda_6]$ of degree $n-1$ if \mbox{$\mu_n\neq0$}.
In~particular, the sigma function $\sigma(u)$ does not depend on the choice of a canonical basis $\{\mathfrak{a},\mathfrak{b}\}$ in~the one-dimensional homology group of the elliptic curve $E$ and is determined by the coefficients~$\lambda_2$, $\lambda_4$, $\lambda_6$ of the defining equation of the elliptic curve $E$.

Let $\wp(u)=-\frac{{\rm d}^2}{{\rm d}u^2}\log\sigma(u)$, which is the Weierstrass elliptic function.
Then the following relations hold (cf.~\cite{H})
\begin{gather*}
\wp(u)-\wp(\omega')=\mbox{al}_1(u)^2,\qquad
\wp(u)-\wp(\omega'+\omega'')=\mbox{al}_2(u)^2,\qquad
\wp(u)-\wp(\omega'')=\mbox{al}_3(u)^2,
\end{gather*}
where
\begin{gather*}
\mbox{al}_1(u)=\frac{\exp(\eta'u)\sigma(\omega'-u)}{\sigma(u)\sigma(\omega')}, \qquad \mbox{al}_2(u)=\frac{\exp(\eta'u+\eta''u)\sigma(\omega'+\omega''-u)}{\sigma(u)\sigma(\omega'+\omega'')},
\\
\mbox{al}_3(u)=\frac{\exp(\eta''u)\sigma(\omega''-u)}{\sigma(u)\sigma(\omega'')}.
\end{gather*}
The functions $\mbox{al}_i(u)$, $i=1,2,3$, are called \textit{Weierstrass elliptic al functions} (cf.~\cite{K.Weierstrass-1894}).

\subsection{The two-dimensional sigma function}

We set
\begin{gather*}
M_2(x)=x^5+\lambda_2x^4+\lambda_4x^3+\lambda_6x^2+\lambda_8x+\lambda_{10},\qquad
\lambda_i\in\mathbb{C}.
\end{gather*}
We assume that $M_2(x)$ has no multiple root. We~consider the nonsingular hyperelliptic curve of genus $2$
\begin{gather*}
V=\big\{(x,y)\in\mathbb{C}^2\mid y^2=M_2(x)\big\}.
\end{gather*}
For $(x,y)\in V$, let
\begin{gather*}
\omega_1=-\frac{x}{2y}{\rm d}x,\qquad
\omega_3=-\frac{1}{2y}{\rm d}x,
\end{gather*}
which are the basis of the vector space of holomorphic one forms on $V$.
Further, let
\begin{gather*}
\eta_1=-\frac{x^2}{2y}{\rm d}x,\qquad
\eta_3=\frac{-\lambda_4x-2\lambda_2x^2-3x^3}{2y}{\rm d}x,
\end{gather*}
which are the meromorphic one forms on $V$ with a pole only at $\infty$.
Let $\{\mathfrak{a}_i,\mathfrak{b}_i\}_{i=1}^2$ be a canonical basis in the one-dimensional homology group of the curve~$V$. We~define the matrices of periods~by
\begin{gather*}
2\omega'=\int_{\mathfrak{a}_j}\omega_i, \qquad 2\omega''=\int_{\mathfrak{b}_j}\omega_i,\qquad
-2\eta'=\int_{\mathfrak{a}_j}\eta_i, \qquad
-2\eta''=\int_{\mathfrak{b}_j}\eta_i.
\end{gather*}
The matrix of normalized periods has the form $\tau=(\omega')^{-1}\omega''$.
Let $\delta=\tau\delta'+\delta''$, $\delta',\delta''\in\mathbb{R}^2,$ be the vectors of Riemann's constants with respect to $(\{\mathfrak{a}_i,\mathfrak{b}_i\}_{i=1}^2,\infty)$
and $\delta:={}^t({}^t\delta',{}^t\delta'')$.
Then we have $\delta'={}^t\big(\frac{1}{2},\frac{1}{2}\big)$ and $\delta''={}^t\big(1,\frac{1}{2}\big)$ (cf.~\cite{Fay-1973,Mumford-1983}).
The sigma function $\sigma({\bf u})$ associated with $V$, ${\bf u}={}^t(u_1,u_3)\in\mathbb{C}^2$, is defined by
\begin{gather*}
\sigma({\bf u})=\varepsilon\exp\bigg(\frac{1}{2}{}^t{\bf u}\eta'(\omega')^{-1}{\bf u}\bigg)\theta[\delta]\big((2\omega')^{-1}{\bf u},\tau\big),
\end{gather*}
where $\theta[\delta]({\bf u})$ is the Riemann's theta function with characteristics $\delta$, which is defined by
\begin{gather*}
\theta[\delta]({\bf u})=\sum_{n\in\mathbb{Z}^2}\exp\big\{\pi{\bf i}\,{}^t(n+\delta')\tau(n+\delta')+2\pi{\bf i}\,{}^t(n+\delta')({\bf u}+\delta'')\big\},
\end{gather*}
and $\varepsilon$ is a constant.

\begin{Proposition}[{\cite[Theorem 1.1]{BEL}, \cite[p.~193]{N-2010}}]
For $m_1,m_2\in\mathbb{Z}^2$, let $\Omega=2\omega'm_1+2\omega''m_2$.
Then, for ${\bf u}\in\mathbb{C}^2$, we have
\begin{gather*}
\sigma({\bf u}+\Omega)/\sigma({\bf u}) =(-1)^{2({}^t\delta'm_1-{}^t\delta''m_2)+{}^tm_1m_2}\exp\big\{{}^t(2\eta'm_1+
2\eta''m_2)({\bf u}+\omega'm_1+\omega''m_2)\big\}.
\end{gather*}
\end{Proposition}

We set $\deg\lambda_{2i}=2i$ for $1\le i\le 5$.

\begin{Theorem}[{\cite[Theorem 6.3]{BEL-99-R}, \cite[Theorem~7.7]{BEL-2012}, \cite[Theorem 3]{N-2010}}]
The sigma func\-tion~$\sigma({\bf u})$ is an entire function on $\mathbb{C}^2$ and we have the unique constant $\varepsilon$ such that the series expansion of~$\sigma({\bf u})$ around the origin has the following form:
\begin{gather}
\sigma({\bf u})=\frac{1}{3}u_1^3-u_3+\sum_{n_1+3n_3\ge5}\mu_{n_1,n_3}u_1^{n_1}u_3^{n_3},\label{4.27.1}
\end{gather}
where the coefficient $\mu_{n_1,n_3}$ is a homogeneous polynomial in $\mathbb{Q}[\lambda_2, \lambda_4, \lambda_6, \lambda_8, \lambda_{10}]$
of degree $n_1+3n_3-3$ if $\mu_{n_1,n_3}\neq0$.
\end{Theorem}

We take the constant $\varepsilon$ such that the expansion (\ref{4.27.1}) holds, see the expression for the sigma function above, which involves the constant $\varepsilon$.
Then the sigma function $\sigma({\bf u})$ does not depend on the choice of a canonical basis $\{\mathfrak{a}_i,\mathfrak{b}_i\}_{i=1}^2$ in the one-dimensional homology group of the curve~$V$ and is determined by the coefficients $\lambda_2$, $\lambda_4$, $\lambda_6$, $\lambda_8$, $\lambda_{10}$ of the defining equation of the curve~$V$.

\begin{Remark}
The constant $\varepsilon$ is given explicitly in \cite[p.~906]{E-H-K-K-L-1} and \cite[p.~9]{E-H-K-K-L-P-1}.
\end{Remark}

\subsection[Inversion problem of the elliptic integrals for the Jacobi and Legendre forms]
{Inversion problem of the elliptic integrals for the Jacobi \\and Legendre forms}

Let
\begin{gather*}
f_J(s)=s^4+a_{J,2}s^2+a_{J,4},
\end{gather*}
where $a_{J,2},a_{J,4}\in\mathbb{C}$ satisfying $a_{J,4}\big(a_{J,2}^2-4a_{J,4}\big)\neq0$. We~consider the elliptic curve in the Jacobi form
\begin{gather*}
E_J=\big\{(s,t)\in\mathbb{C}^2\mid t^2=f_J(s)\big\}
\end{gather*}
and the elliptic curve in the Weierstrass normal form
\begin{gather*}
\widetilde{E}_J=\bigg\{(x,y)\in\mathbb{C}^2\,\bigg|\,y^2
=x^3-\bigg(4a_{J,4}+\frac{1}{3}a_{J,2}^2\bigg)x-\frac{8}{3}a_{J,2}a_{J,4}+\frac{2}{27}a_{J,2}^3\bigg\}.
\end{gather*}
We take $a_J\in\mathbb{C}$ such that $f_J(a_J)=0$. We~have the isomorphism
\begin{gather*}
\xi_J\colon\quad E_J\to\widetilde{E}_J,\qquad (s,t)\mapsto(x,y),
\end{gather*}
with
\begin{gather}\label{2021202120212021.7.12.1}
x=\frac{f_J'(a_J)}{s-a_J}+\frac{f_J''(a_J)}{6},\qquad
y=\frac{f_J'(a_J)t}{(s-a_J)^2}.
\end{gather}
Let $\{\mathfrak{a}_J,\mathfrak{b}_J\}$ be a canonical basis in the one-dimensional homology group of the curve $E_J$. We~define the periods by
\begin{gather*}
2\Omega_J'=\int_{\mathfrak{a}_J}\frac{{\rm d}s}{2t},\qquad
2\Omega_J''=\int_{\mathfrak{b}_J}\frac{{\rm d}s}{2t}.
\end{gather*}
We define the period lattice $\Lambda_J=\{2\Omega_J'm_1+2\Omega_J''m_2\mid m_1,m_2\in\mathbb{Z}\}$ and consider the Jacobian variety $\operatorname{Jac}(E_J)=\mathbb{C}/\Lambda_J$.
Let $\sigma_{\widetilde{E}_J}(u)$ be the sigma function associated with the elliptic curve~$\widetilde{E}_J$ and $\wp_{\widetilde{E}_J}=-\frac{{\rm d}^2}{{\rm d}u^2}\log\sigma_{\widetilde{E}_J}$, which is the Weierstrass elliptic function. We~consider the map
\begin{gather*}
I_J\colon\quad E_J\to\operatorname{Jac}(E_J),\qquad
S\mapsto\int_{(a_J,0)}^S\frac{{\rm d}s}{2t}.
\end{gather*}
The map $I_J$ is the isomorphism.
Let $u=I_J(S)$ for $S=(s,t)\in E_J$.

\begin{Proposition}[{\cite[pp.~453, 454]{WW}}]
We have
\begin{gather*}
s=a_J+\frac{f_J'(a_J)}{\wp_{\widetilde{E}_J}(u)-f_J''(a_J)/6},\qquad
t=-\frac{f_J'(a_J)\wp_{\widetilde{E}_J}'(u)}{2(\wp_{\widetilde{E}_J}(u)-f_J''(a_J)/6)^2}.
\end{gather*}
\end{Proposition}

\begin{proof}
Since the pullback of the holomorphic one form $-{\rm d}x/(2y)$ on $\widetilde{E}_J$ with respect to $\xi_J$ is~${\rm d}s/(2t)$, we have
\begin{gather*}
u=\int_{(a_J,0)}^{(s,t)}\frac{{\rm d}s}{2t}=\int_{\infty}^{(x,y)}\bigg({-}\frac{{\rm d}x}{2y}\bigg).
\end{gather*}
We have
\begin{gather}
x=\wp_{\widetilde{E}_J}(u),\qquad
y=-\frac{\wp_{\widetilde{E}_J}'(u)}2.\label{2021202120212021.7.12.2}
\end{gather}
From (\ref{2021202120212021.7.12.1}) and (\ref{2021202120212021.7.12.2}), we obtain the statement of the proposition.
\end{proof}

We consider the elliptic curve in the Legendre form
\begin{gather*}
E_L=\big\{(s,t)\in\mathbb{C}^2\mid t^2=s(s-b_L)(s-c_L)\big\},
\end{gather*}
where $b_L,c_L\in\mathbb{C}$ satisfying $b_Lc_L(b_L-c_L)\neq0$,
and the elliptic curve in the Weierstrass normal form
\begin{gather*}
\widetilde{E}_L=\bigg\{(x,y)\in\mathbb{C}^2\,\bigg|\,y^2=x^3-\bigg(\frac{b_L^2+c_L^2-b_Lc_L}{3}\bigg)x
+\frac{(2b_L-c_L)(b_L+c_L)(2c_L-b_L)}{27}\bigg\}.
\end{gather*}
We have the isomorphism
\begin{gather*}
\xi_L\colon\quad E_L\to\widetilde{E}_L,\qquad (s,t)\mapsto(x,y),
\end{gather*}
with
\begin{gather}
x=s-\frac{b_L+c_L}{3},\qquad y=t.\label{2021202120212021.7.13.1}
\end{gather}
Let $\{\mathfrak{a}_L,\mathfrak{b}_L\}$ be a canonical basis in the one-dimensional homology group of the curve $E_L$. We~define the periods by
\begin{gather*}
2\Omega_L'=\int_{\mathfrak{a}_L}\bigg({-}\frac{{\rm d}s}{2t}\bigg),\qquad
2\Omega_L''=\int_{\mathfrak{b}_L}\bigg({-}\frac{{\rm d}s}{2t}\bigg).
\end{gather*}
We define the period lattice $\Lambda_L=\{2\Omega_L'm_1+2\Omega_L''m_2\mid m_1,m_2\in\mathbb{Z}\}$ and consider the Jacobian variety $\operatorname{Jac}(E_L)=\mathbb{C}/\Lambda_L$.
Let $\sigma_{\widetilde{E}_L}(u)$ be the sigma function associated with the elliptic curve~$\widetilde{E}_L$ and $\wp_{\widetilde{E}_L}=-\frac{{\rm d}^2}{{\rm d}u^2}\log\sigma_{\widetilde{E}_L}$, which is the Weierstrass elliptic function. We~consider the map
\begin{gather*}
I_L\colon\quad E_L\to\operatorname{Jac}(E_L),\qquad
S\mapsto\int_{\infty}^S\bigg({-}\frac{{\rm d}s}{2t}\bigg).
\end{gather*}
The map $I_L$ is the isomorphism.
Let $u=I_L(S)$ for $S=(s,t)\in E_L$.

\begin{Proposition}
We have
\begin{gather*}
s=\wp_{\widetilde{E}_L}(u)+\frac{b_L+c_L}{3},\qquad
t=-\frac{\wp_{\widetilde{E}_L}'(u)}2.
\end{gather*}
\end{Proposition}

\begin{proof}
Since the pullback of the holomorphic one form $-{\rm d}x/(2y)$ on $\widetilde{E}_L$ with respect to $\xi_L$ is $-{\rm d}s/(2t)$, we have
\begin{gather*}
u=\int_{\infty}^{(s,t)}\bigg({-}\frac{{\rm d}s}{2t}\bigg)=\int_{\infty}^{(x,y)}\bigg({-}\frac{{\rm d}x}{2y}\bigg).
\end{gather*}
We have
\begin{gather}
x=\wp_{\widetilde{E}_L}(u),\qquad
y=-\frac{\wp_{\widetilde{E}_L}'(u)}2.\label{2021202120212021.7.13.2}
\end{gather}
From (\ref{2021202120212021.7.13.1}) and (\ref{2021202120212021.7.13.2}), we obtain the statement of the proposition.
\end{proof}

\section{The curve of genus 2 and the elliptic curves}\label{2021.11.22.1111}

\begin{Theorem}[{\cite{J,SV}}]\label{3.16}
Let $H$ be a hyperelliptic curve of genus $2$. There exist an elliptic curve~$W$ and a morphism $H\to W$ of degree $2$ if and only if $H$ is isomorphic to the hyperelliptic curve~$H'$ of genus $2$ defined by
\begin{gather}
H'=\big\{(s,t)\in\mathbb{C}^2\mid t^2=\big(s^2-1\big)\big(s^2-e_1^2\big)\big(s^2-e_2^2\big)\big\},\label{3.12.2}
\end{gather}
where $e_1,e_2\in\mathbb{C}$ satisfying
\begin{gather}
e_1^2\neq0, 1,\qquad
e_2^2\neq0, 1,\qquad
e_1^2\neq e_2^2.\label{2021.11.12.1}
\end{gather}
\end{Theorem}

In appendix, we give a proof of this theorem.
In~\cite[Section~1.3.1]{BCMS2021}, \cite[Chapter~14]{Cassels}, \cite[Section~4]{F2}, and \cite[Theorem~4.2]{W} , the elliptic curves
\begin{gather*}
W_1=\big\{(s,t)\in\mathbb{C}^2\mid t^2=(s-1)\big(s-e_1^2\big)\big(s-e_2^2\big)\big\},
\\
W_2=\big\{(s,t)\in\mathbb{C}^2\mid t^2=(1-s)\big(1-e_1^2s\big)\big(1-e_2^2s\big)\big\}
\end{gather*}
and the morphisms
\begin{gather*}
\phi_1\colon\quad H'\to W_1,\qquad (s,t)\mapsto\big(s^2,t\big),\qquad
\phi_2\colon\quad H'\to W_2,\qquad (s,t)\mapsto\big(1/s^2, t/s^3\big)
\end{gather*}
are considered.
It is well known that the Jacobian variety of $H'$ is isogenous to $W_1\times W_2$ (cf.~\cite[Chapter~14]{Cassels}, \cite[Section~4]{F2}, \cite[Section~1]{FK}, \cite[Section~2]{Ku}, and \cite[Theorem~4.2]{W}). In~\cite[Section~7]{P}, many examples of the curves of genus 2 whose Jacobian varieties are isogenous to the direct products of elliptic curves are given.
These examples are important from the point of view of problems in number theory.

Throughout this paper, for a complex number $z$, we denote a complex number whose square is $z$ by $\sqrt{z}$.
For $\alpha, \beta\in\mathbb{C}$ satisfying
\begin{gather}
\alpha^2,\beta^2\neq0,1,\qquad
\alpha^2\neq\beta^2,\qquad
\alpha^2\beta^2\neq1,\label{2021.11.12.2}
\end{gather}
we consider the nonsingular hyperelliptic curve of genus $2$\footnote{The hyperelliptic curve of genus 2 defined by
\begin{gather*}
y^2=x(1-x)\big(1-\nu_1^2x\big)\big(1-\nu_2^2x\big)\big(1-\nu_3^2x\big),\qquad \nu_1, \nu_2,\nu_3\in\mathbb{C},
\end{gather*}
is called \textit{Richelot normal form} (cf.~\cite[p.~450]{O-Bolza-1887}, \cite[pp.~97, 236]{BEL-2012}).}
\begin{gather}
V=\big\{(x,y)\in\mathbb{C}^2\mid y^2=x(x-1)\big(x-\alpha^2\big)\big(x-\beta^2\big)\big(x-\alpha^2\beta^2\big)\big\}.\label{4.7.1}
\end{gather}
This curve~$V$ is considered in \cite{BE0, BE, BE2, BCMS2021, BEL, BEL-2012, Enolskii-Salerno-1996}.
For $i=2,4,6,8$, we define $\lambda_i$ such that the following relation holds:
\begin{gather*}
x(x-1)\big(x-\alpha^2\big)\big(x-\beta^2\big)\big(x-\alpha^2\beta^2\big)=x^5+\lambda_2x^4+\lambda_4x^3+\lambda_6x^2+\lambda_8x,
\end{gather*}
i.e., we set
\begin{gather*}
\lambda_2=-1-\alpha^2-\beta^2-\alpha^2\beta^2,\qquad
\lambda_4=\alpha^2+\beta^2+2\alpha^2\beta^2+\alpha^4\beta^2+\alpha^2\beta^4,
\\
\lambda_6=-\alpha^4\beta^4-\alpha^2\beta^4-\alpha^4\beta^2-\alpha^2\beta^2,\qquad
\lambda_8=\alpha^4\beta^4.
\end{gather*}

\begin{Proposition}[{\cite[p.~235]{BE0}, \cite[p.~3442]{BE}, \cite[Proposition~1.10]{BCMS2021}, \cite[Theorem~14.1.1]{Cassels}}]
Given $e_1, e_2\in\mathbb{C}$ satisfying \eqref{2021.11.12.1}, the curve $H'$ is isomorphic to the curve~$V$ with
\begin{gather}
\alpha=\sqrt{\frac{(e_1+1)(e_2+1)}{(e_1-1)(e_2-1)}},\qquad
\beta=\sqrt{\frac{(e_1-1)(e_2+1)}{(e_1+1)(e_2-1)}}.\label{2021.11.12.3}
\end{gather}
Conversely, given $\alpha, \beta\in\mathbb{C}$ satisfying \eqref{2021.11.12.2}, the curve~$V$ is isomorphic to the curve $H'$ with
\begin{gather}
e_1=\frac{\alpha+\beta}{\alpha-\beta},\qquad
e_2=\frac{\alpha\beta+1}{\alpha\beta-1}.\label{2021.11.12.4}
\end{gather}
\end{Proposition}

\begin{Proposition}[{\cite[p.~235]{BE0}, \cite[p.~3442]{BE}}]\label{2021.11.22.3}
Given $e_1, e_2\in\mathbb{C}$ satisfying \eqref{2021.11.12.1}, an isomorphism from $H'$ to $V$ with \eqref{2021.11.12.3} is given by
\begin{gather*}
\zeta\colon\quad H'\to V,\qquad (s,t)\mapsto (x,y)
\end{gather*}
with
\begin{gather*}
(x,y)=\bigg(\frac{(e_2+1)(s+1)}{(e_2-1)(s-1)}, \frac{4(e_2+1)^2t}{(e_2-1)^3\sqrt{e_1^2-1}(s-1)^3}\bigg).
\end{gather*}
Furthermore, given $\alpha, \beta\in\mathbb{C}$ satisfying \eqref{2021.11.12.2}, an isomorphism from $V$ to $H'$ with \eqref{2021.11.12.4} is given by
\begin{gather*}
\widetilde{\zeta}\colon\quad V\to H',\qquad (x,y)\mapsto(s,t)
\end{gather*}
with
\begin{gather*}
(s,t)=\bigg(\frac{x+\alpha\beta}{x-\alpha\beta},
\frac{8\alpha\beta\sqrt{\alpha\beta}y}{(\alpha\beta-1)(\alpha-\beta)(x-\alpha\beta)^3}\bigg).
\end{gather*}
\end{Proposition}

The elliptic curve $W_1$ with (\ref{2021.11.12.4}) is isomorphic to the elliptic curve in Legendre form
\begin{gather}
E_1=\bigg\{(X,Y)\in\mathbb{C}^2\,\bigg|\,
Y^2=X(X-1)\bigg(X-\frac{(\alpha-\beta)^2}{(\alpha\beta-1)^2}\bigg)\bigg\}\label{4.6.1a}
\end{gather}
by the morphism
\begin{gather*}
\xi_1\colon\quad W_1\to E_1,\qquad (s,t)\mapsto(X,Y)
\end{gather*}
with
\begin{gather*}
(X,Y)=\bigg(\frac{(\alpha-\beta)^2(s-1)}{4\alpha\beta}, \frac{(\alpha-\beta)^3t}{8\alpha\beta\sqrt{\alpha\beta}}\bigg).
\end{gather*}
The elliptic curve $W_2$ with (\ref{2021.11.12.4}) is isomorphic to the elliptic curve in Legendre form
\begin{gather}
E_2=\bigg\{(X,Y)\in\mathbb{C}^2\,\bigg|\,
Y^2=X(X-1)\bigg(X-\frac{(\alpha+\beta)^2}{(\alpha\beta+1)^2}\bigg)\bigg\}\label{4.6.2a}
\end{gather}
by the morphism
\begin{gather*}
\xi_2\colon\quad W_2\to E_2,\qquad (s,t)\mapsto(X,Y)
\end{gather*}
with
\begin{gather*}
(X,Y)=\bigg({-}\frac{(\alpha+\beta)^2(s-1)}{4\alpha\beta}, -\frac{(\alpha+\beta)^2(\alpha-\beta)(\alpha\beta-1)t}
{8\alpha\beta\sqrt{\alpha\beta}(\alpha\beta+1)}\bigg).
\end{gather*}

We consider the maps
\begin{gather*}
\varphi_i=\xi_i\circ\phi_i\circ\widetilde{\zeta}\colon\quad V\to E_i,\qquad i=1,2.
\end{gather*}
Then the maps $\varphi_i$, $i=1,2$, are described by
\begin{gather}
\varphi_1\colon\quad V\to E_1,\qquad
(x,y)\mapsto(X,Y)=\bigg(\frac{(\alpha-\beta)^2x}{(x-\alpha\beta)^2}, \frac{(\alpha-\beta)^2y}{(\alpha\beta-1)(x-\alpha\beta)^3}\bigg), \label{3.16.11}
\\
\varphi_2\colon\quad V\to E_2,\qquad (x,y)\mapsto(X,Y)=\bigg(\frac{(\alpha+\beta)^2x}{(x+\alpha\beta)^2}, \frac{-(\alpha+\beta)^2y}{(\alpha\beta+1)(x+\alpha\beta)^3}\bigg).\label{3.16.222}
\end{gather}

\begin{Remark}
For $P=(x,y)\in V\backslash\{\infty\}$ such that $x=\alpha\beta$, we have $\varphi_1(P)=\infty$.
For $P=(x,y)\in V\backslash\{\infty\}$ such that $x=-\alpha\beta$, we have $\varphi_2(P)=\infty$.
Furthermore, we have $\varphi_i(\infty)=(0,0)$ for $i=1,2$.
\end{Remark}

Let
\begin{gather*}
\kappa_1=\frac{\sqrt{-1} (\alpha-\beta)}{\sqrt{\big(1-\alpha^2\big)\big(1-\beta^2\big)}},\qquad
\kappa_2=\frac{\sqrt{-1} (\alpha+\beta)}{\sqrt{\big(1-\alpha^2\big)\big(1-\beta^2\big)}}.
\end{gather*}

\begin{Remark}
In~\cite[Section~7.1]{BE0}, \cite[Section~3.1.1]{BE}, \cite[Chapter~6]{BEL-2012}, \cite{Enolskii-Salerno-1996}, the elliptic curves
\begin{gather*}
E_+=\bigg\{(X_+, Y_+)\in\mathbb{C}^2\mid Y_+^2=X_+(1-X_+)\big(1-\kappa_1^2X_+\big)\bigg\},
\\
E_-=\bigg\{(X_-, Y_-)\in\mathbb{C}^2\mid Y_-^2=X_-(1-X_-)\big(1-\kappa_2^2X_-\big)\bigg\}
\end{gather*}
and the morphisms
\begin{gather*}
\pi_+\colon\quad V\to E_+,\qquad (x,y)\mapsto (X_+, Y_+)
\end{gather*}
with
\begin{gather*}
(X_+, Y_+)=\bigg(\frac{\big(1-\alpha^2\big)\big(1-\beta^2\big)x}{\big(x-\alpha^2\big)\big(x-\beta^2\big)}, \frac{-\sqrt{\big(1-\alpha^2\big)\big(1-\beta^2\big)}(x-\alpha\beta)y}{\big(x-\alpha^2\big)^2\big(x-\beta^2\big)^2}\bigg),
\end{gather*}
\begin{gather*}
\pi_-\colon\quad V\to E_-,\qquad (x,y)\mapsto (X_-, Y_-)
\end{gather*}
with
\begin{gather*}
(X_-, Y_-)=\bigg(\frac{\big(1-\alpha^2\big)\big(1-\beta^2\big)x}{\big(x-\alpha^2\big)\big(x-\beta^2\big)}, \frac{-\sqrt{\big(1-\alpha^2\big)\big(1-\beta^2\big)}(x+\alpha\beta)y}{\big(x-\alpha^2\big)^2\big(x-\beta^2\big)^2}\bigg)
\end{gather*}
are considered. We~have the isomorphisms
\begin{gather*}
\xi_+\colon\quad E_1\to E_+,\qquad (X,Y)\mapsto(X_+,Y_+)
\end{gather*}
with
\begin{gather*}
(X_+,Y_+)=\bigg(\frac{X}{\kappa_1^2(X-1)},
\frac{(1-\alpha\beta)\sqrt{\big(1-\alpha^2\big)\big(1-\beta^2\big)}Y}{(\alpha-\beta)^2(X-1)^2}\bigg),
\end{gather*}
and
\begin{gather*}
\xi_-\colon \quad E_2\to E_-,\qquad (X,Y)\mapsto(X_-,Y_-)
\end{gather*}
with
\begin{gather*}
(X_-,Y_-)=\bigg(\frac{X}{\kappa_2^2(X-1)},
\frac{(1+\alpha\beta)\sqrt{\big(1-\alpha^2\big)\big(1-\beta^2\big)}Y}{(\alpha+\beta)^2(X-1)^2}\bigg).
\end{gather*}
Furthermore, we have $\pi_+=\xi_+\circ\varphi_1$ and $\pi_-=\xi_-\circ\varphi_2$.
\end{Remark}

\begin{Remark}
In~\cite[Lemma 1.7 and Proposition 1.8]{BCMS2021}, for $i=1,2$, the elliptic curves
\begin{gather*}
\mathcal{E}_i=\big\{(\mathcal{X}_i,\mathcal{Y}_i)\in\mathbb{C}^2\mid\mathcal{Y}_i^2
=\mathcal{X}_i(\mathcal{X}_i-1)\big(\mathcal{X}_i-\kappa_i^2\big)\big\}
\end{gather*}
and the morphisms
\begin{gather*}
\pi_{\mathcal{E}_i}\colon\quad V\to \mathcal{E}_i,\qquad
(x,y)\mapsto (\mathcal{X}_i,\mathcal{Y}_i)
\end{gather*}
with
\begin{gather*}
(\mathcal{X}_i,\mathcal{Y}_i)=\bigg(\frac{\big(x-\alpha^2\big)\big(x-\beta^2\big)}{\big(1-\alpha^2\big)\big(1-\beta^2\big)x},
\frac{\big(x+(-1)^i\alpha\beta\big)y}{\big(1-\alpha^2\big)\big(1-\beta^2\big)\sqrt{\big(1-\alpha^2\big)\big(1-\beta^2\big)}x^2}\bigg)
\end{gather*}
are considered.
For $i=1,2$, we have the isomorphisms
\begin{gather*}
\overline{\xi}_i\colon\quad E_i\to\mathcal{E}_i,\qquad (X,Y)\mapsto(\mathcal{X}_i,\mathcal{Y}_i)
\end{gather*}
with
\begin{gather*}
(\mathcal{X}_i,\mathcal{Y}_i)=\left(\frac{\kappa_i^2(X-1)}{X},
\frac{\kappa_i^2(1+(-1)^i\alpha\beta)Y}{\sqrt{\big(1-\alpha^2\big)\big(1-\beta^2\big)}X^2}\right)
\end{gather*}
and $\pi_{\mathcal{E}_i}=\overline{\xi}_i\circ\varphi_i$.
\end{Remark}

\begin{Remark}
In~\cite[Section~1.3.4]{BCMS2021}, \cite[Section~6.1]{B-L-S-2013}, \cite[Section~3.3]{J.P.Serre-2020}, given two elliptic curves in Legendre form satisfying a condition, a method to construct a hyperelliptic curve of genus~2 whose Jacobian variety is
$(2,2)$-isogenous to the direct product of the two elliptic curves is introduced, which is called \textit{Legendre's glueing method}.
\end{Remark}

We consider the hyperelliptic curve~$V$ defined by (\ref{4.7.1}) and the elliptic curves $E_1$ and $E_2$ defined by (\ref{4.6.1a}) and (\ref{4.6.2a}), respectively. We~define the period lattice $\Lambda=\big\{2\omega'm_1+2\omega''m_2\mid m_1,m_2\in\mathbb{Z}^2\big\}$ and consider the Jacobian variety $\operatorname{Jac}(V)=\mathbb{C}^2/\Lambda$.
For $i=1,2$, let $\big\{\mathfrak{a}^{(i)},\mathfrak{b}^{(i)}\big\}$ be a canonical basis in the one-dimensional homology group of the curve $E_i$.
For $i=1,2$, we consider the holomorphic one form on $E_i$
\begin{gather*}
\Omega_i=-\frac{{\rm d}X}{2Y}.
\end{gather*}
We define the periods by
\begin{gather*}
2\Omega_i'=\int_{\mathfrak{a}^{(i)}}\Omega_i,\qquad
2\Omega_i''=\int_{\mathfrak{b}^{(i)}}\Omega_i.
\end{gather*}
We define the period lattice $\Lambda_i=\big\{2\Omega_i'm_1+2\Omega_i''m_2\mid m_1,m_2\in\mathbb{Z}\big\}$ and consider the Jacobian variety $\operatorname{Jac}(E_i)=\mathbb{C}/\Lambda_i$.
Let $\pi\colon \mathbb{C}^2\to\operatorname{Jac}(V)$ and $\pi_i\colon \mathbb{C}\to\operatorname{Jac}(E_i)$ be the natural projections for $i=1,2$.

In~\cite{P}, the curves of genus 2 defined by
\begin{gather*}
C_{48,1}=\big\{(s,t)\in\mathbb{C}^2\mid t^2=(s-2)(s+2)\big(s^4-10s^2-3\big)\big\},
\\
C_{48,2}=\big\{(s,t)\in\mathbb{C}^2\mid t^2=\big(s^2+s+2\big)\big(s^4+3s^3+21s^2-27s+18\big)\big\}
\end{gather*}
are considered.
The groups of torsion points over $\mathbb{Q}$ in the Jacobian varieties of $C_{48,1}$ and $C_{48,2}$ are isomorphic to the cyclic groups of order 48~\cite[Theorem~6.2]{P}.
The curve $C_{48,1}$ is isomorphic to the curve of genus 2
\begin{gather*}
\widetilde{C}_{48,1}=\bigg\{(s,t)\in\mathbb{C}^2\,\bigg|\,
t^2=\big(s^2-1\big)\bigg(s^2-\frac{a_{48,1}^2}{4}\bigg)\bigg(s^2-\frac{b_{48,1}^2}{4}\bigg)\bigg\}
\end{gather*}
by the morphism
\begin{gather*}
C_{48,1}\to\widetilde{C}_{48,1},\qquad (s,t)\mapsto(s/2, t/8),
\end{gather*}
where
\begin{gather*}
a_{48,1}=\sqrt{5+2\sqrt{7}},\qquad b_{48,1}=\sqrt{5-2\sqrt{7}}.
\end{gather*}
The curve $\widetilde{C}_{48,1}$ is isomorphic to the curve~$V$ defined by (\ref{4.7.1}) with
\begin{gather}
\alpha=\sqrt{\frac{(a_{48,1}+2)(b_{48,1}+2)}{(a_{48,1}-2)(b_{48,1}-2)}},\qquad
\beta=\sqrt{\frac{(a_{48,1}-2)(b_{48,1}+2)}{(a_{48,1}+2)(b_{48,1}-2)}}.
\label{202120212021202120212021.7.19}
\end{gather}
Then the Jacobian variety $\operatorname{Jac}(V)$ is isogenous to $E_1\times E_2$ defined by (\ref{4.6.1a}) and (\ref{4.6.2a}) with (\ref{202120212021202120212021.7.19}).
The curve $C_{48,2}$ is isomorphic to the curve of genus 2
\begin{gather*}
\widetilde{C}_{48,2}=\bigg\{(s,t)\in\mathbb{C}^2\,\bigg|\,
t^2=\big(s^2-1\big)\bigg(s^2-\frac{a_{48,2}^2}{6}\bigg)\bigg(s^2-\frac{b_{48,2}^2}{6}\bigg)\bigg\}
\end{gather*}
by the morphism
\begin{gather*}
C_{48,2}\to\widetilde{C}_{48,2},\qquad
(s,t)\mapsto\bigg(\frac{\sqrt{-7}(s+1)}{s-3}, \frac{56\sqrt{-1}t}{3(s-3)^3}\bigg),
\end{gather*}
where
\begin{gather*}
a_{48,2}=\sqrt{-7-7\sqrt{-7}},\qquad
b_{48,2}=\sqrt{-7+7\sqrt{-7}}.
\end{gather*}
The curve $\widetilde{C}_{48,2}$ is isomorphic to the curve~$V$ defined by (\ref{4.7.1}) with
\begin{gather}
\alpha=\sqrt{\frac{\big(a_{48,2}+\sqrt{6}\big)\big(b_{48,2}+\sqrt{6}\big)}{\big(a_{48,2}-\sqrt{6}\big)\big(b_{48,2}-\sqrt{6}\big)}},\qquad
\beta=\sqrt{\frac{\big(a_{48,2}-\sqrt{6}\big)\big(b_{48,2}+\sqrt{6}\big)}{\big(a_{48,2}+\sqrt{6}\big)\big(b_{48,2}-\sqrt{6}\big)}}.
\label{202120212021202120212021.7.19.1}
\end{gather}
Then the Jacobian variety $\operatorname{Jac}(V)$ is isogenous to $E_1\times E_2$ defined by (\ref{4.6.1a}) and (\ref{4.6.2a}) with (\ref{202120212021202120212021.7.19.1}).

\section[The map from Jac(V) to Jac(E\_i)]
{The map from $\boldsymbol{\operatorname{Jac}(V)}$ to $\boldsymbol{\operatorname{Jac}(E_i)}$}

We denote the coordinates in $\mathbb{C}^2$ and $\mathbb{C}$ by ${}^t(u_1,u_3)$ and $u$, respectively.
We take the following two points on $V$:
\begin{gather*}
O_1=\big(\alpha\beta, \alpha\beta\sqrt{\alpha\beta}(\alpha\beta-1)(\alpha-\beta)\big),\qquad
O_2=\big({-}\alpha\beta, \alpha\beta\sqrt{-\alpha\beta}(\alpha\beta+1)(\alpha+\beta)\big).
\end{gather*}
We have $\varphi_i(O_i)=\infty$ for $i=1,2$.
Let ${\boldsymbol \omega}={}^t(\omega_1,\omega_3)$. We~consider the holomorphic maps
\begin{gather*}
I_1\colon\quad V\to\operatorname{Jac}(V),\qquad P\mapsto \int_{O_1}^P{\boldsymbol \omega},
\\
I_2\colon\quad V\to\operatorname{Jac}(V),\qquad P\mapsto \int_{O_2}^P{\boldsymbol \omega}.
\end{gather*}
Note that the elliptic curves $E_1$ and $E_2$ have the group structures (cf.~\cite[Chapter~III, Section~2]{silverman}).
Since $\operatorname{Jac}(V)$ is the Albanese variety of $V$, we have the unique holomorphic map
\begin{gather*}
\varphi_{1,*}\colon\ \operatorname{Jac}(V)\to E_1
\end{gather*}
such that $\varphi_{1,*}$ is the group homomorphism and $\varphi_1=\varphi_{1,*}\circ I_1$ and
we have the unique holomorphic map
\begin{gather*}
\varphi_{2,*}\colon\ \operatorname{Jac}(V)\to E_2
\end{gather*}
such that $\varphi_{2,*}$ is the group homomorphism and $\varphi_2=\varphi_{2,*}\circ I_2$ (cf.~\cite[Proposition~6.1, p.~104]{Milne}).
\begin{gather*}
\xymatrix{
\operatorname{Jac} (V) \ar[rd]^{\varphi_{1,*}} \\
V \ar[u]^{I_1} \ar[r]^{\varphi_1} &E_1,}\qquad
\xymatrix{
\operatorname{Jac} (V) \ar[rd]^{\varphi_{2,*}} \\
V \ar[u]^{I_2} \ar[r]^{\varphi_2} &E_2.}
\end{gather*}
For $i=1,2$, we consider the maps
\begin{gather*}
J_i\colon\quad E_i\to\operatorname{Jac}(E_i),\qquad S\mapsto\int_{\infty}^S\Omega_i.
\end{gather*}
It is well known that the map $J_i$ is the isomorphism.
We consider the maps
\begin{gather*}
\psi_{i,*}=J_i\circ \varphi_{i,*}\colon\quad \operatorname{Jac}(V)\to\operatorname{Jac}(E_i),\qquad i=1,2.
\end{gather*}
For $i=1,2$, the map $\psi_{i,*}$ is the holomorphic map between the complex manifolds $\operatorname{Jac}(V)$ and~$\operatorname{Jac}(E_i)$, and the map $\psi_{i,*}$ is the homomorphism between the groups $\operatorname{Jac}(V)$ and $\operatorname{Jac}(E_i)$.

\begin{Proposition}\label{3.6.1}
For $i=1,2$, we consider the maps
\begin{gather*}
\widetilde{\psi}_{i,*}\colon\quad \mathbb{C}^2\to\mathbb{C},\qquad {}^t(u_1,u_3)\mapsto a_iu_1+b_iu_3,
\end{gather*}
where
\begin{gather}
a_1=1-\alpha\beta,\qquad
b_1=\alpha\beta(1-\alpha\beta),\qquad
a_2=\alpha\beta+1,\qquad
b_2=-\alpha\beta(\alpha\beta+1).\label{3.8.123}
\end{gather}
Then we have $\psi_{i,*}\circ\pi=\pi_i\circ\widetilde{\psi}_{i,*}$, i.e., we have the following commutative diagram:
\begin{gather*}
 \begin{CD}
 \mathbb{C}^2 @>{\widetilde{\psi}_{i,*}}>> \mathbb{C} \\
 @V{\pi}VV @V{\pi_i}VV \\
 \operatorname{Jac}(V) @>{\psi_{i,*}} >> \operatorname{Jac}(E_i).
 \end{CD}
\end{gather*}
\end{Proposition}

\begin{proof}
Since $\psi_{1,*}$ is holomorphic and the group homomorphism, it is well known that there exist complex numbers $a_1$ and $b_1$ such that
$\psi_{1,*}\circ\pi=\pi_1\circ\widetilde{\psi}_{1,*}$, where
\begin{gather*}
\widetilde{\psi}_{1,*}\colon\quad\mathbb{C}^2\to\mathbb{C},\qquad {}^t(u_1,u_3)\mapsto a_1u_1+b_1u_3,
\end{gather*}
(e.g., \cite[Proposition~1.2.1]{B-H-2004}). We~consider the following commutative diagram:
\begin{gather*}
\xymatrix{
\operatorname{Jac} (V) \ar[r]^{\psi_{1,*}} & \operatorname{Jac}(E_1) \ar[d]^{J_1^{-1}} \\
V \ar[u]^{I_1} \ar[r]^{\varphi_1} &E_1.
}
\end{gather*}
We consider the holomorphic one forms ${\rm d}u$ on $\operatorname{Jac}(E_1)$ and ${\rm d}u_1, {\rm d}u_3$ on $\operatorname{Jac}(V)$.
The pullback of the holomorphic one form $\Omega_1$ on $E_1$ with respect to $J_1^{-1}$ is ${\rm d}u$.
The pullback of ${\rm d}u$ with respect to $\psi_{1,*}$ is $a_1{\rm d}u_1+b_1{\rm d}u_3$.
The pullback of $a_1{\rm d}u_1+b_1{\rm d}u_3$ with respect to $I_1$ is $a_1\omega_1+b_1\omega_3$.
Thus, the pullback of $\Omega_1$ with respect to $J_1^{-1}\circ \psi_{1,*}\circ I_1$ is $a_1\omega_1+b_1\omega_3$.
On the other hand, from the definition of $\varphi_1$ and the direct calculations, we find that the pullback of $\Omega_1$ with respect to $\varphi_1$ is $(1-\alpha\beta)\omega_1+\alpha\beta(1-\alpha\beta)\omega_3$.
Since $\omega_1$ and $\omega_3$ are linearly independent, we obtain $a_1=1-\alpha\beta$ and $b_1=\alpha\beta(1-\alpha\beta)$. We~can obtain the defining equation of the elliptic curve $E_2$ and the map~$\varphi_2$ by changing $\beta$ with $-\beta$ in the defining equation of the elliptic curve $E_1$ and the map $\varphi_1$ (see~(\ref{4.6.1a}), (\ref{4.6.2a}), (\ref{3.16.11}), and (\ref{3.16.222})).
On the other hand, when we change $\beta$ with $-\beta$ in the defining equation of the curve~$V$, the defining equation of $V$ does not change (see (\ref{4.7.1})). Therefore we can obtain $a_2$ and $b_2$ by changing $\beta$ with $-\beta$ in the expressions of~$a_1$ and~$b_1$, respectively.
Thus we have $a_2=\alpha\beta+1$ and $b_2=-\alpha\beta(\alpha\beta+1)$.
\end{proof}

Let $\mbox{Sym}^2(V)$ be the symmetric square of $V$.
Then $\mbox{Sym}^2(V)$ is a complex manifold of dimension 2. We~consider the holomorphic map
\begin{gather*}
I\colon\quad \mbox{Sym}^2(V)\to\operatorname{Jac}(V),\qquad (P,Q)\mapsto \int_{\infty}^{P}{\boldsymbol \omega}+\int_{\infty}^{Q}{\boldsymbol \omega}.
\end{gather*}
Let $\sigma({\bf u})$, ${\bf u}={}^t(u_1,u_3)\in\mathbb{C}^2$, be the sigma function associated with the curve~$V$, $\wp_{j,k}=-\partial_{u_k}\partial_{u_j}\log\sigma$, and $\wp_{j,k,\ell}=\partial_{u_{\ell}}\wp_{j,k}$, where $\partial_{u_j}=\frac{\partial}{\partial u_j}$.
For $i=1,2$, let $\sigma_{E_i}(u)$ be the sigma function associated with the elliptic curve $E_i$, $\wp_{E_i}=-\frac{{\rm d}^2}{{\rm d}u^2}\log\sigma_{E_i}$, which is the Weierstrass elliptic function, and $\wp_{E_i}'=\frac{\rm d}{{\rm d}u}\wp_{E_i}$.
Let $f_i({\bf u})=\wp_{E_i}(\psi_{i,*}({\bf u}))$ for ${\bf u}\in\operatorname{Jac}(V)$ and $i=1,2$.
Then the functions $f_1({\bf u})$ and $f_2({\bf u})$ are the meromorphic functions on $\operatorname{Jac}(V)$.
When we regard $f_1({\bf u})$ and $f_2({\bf u})$ as meromorphic functions on $\mathbb{C}^2$, from Proposition \ref{3.6.1}, we have
\begin{gather*}
f_1({\bf u})=\wp_{E_1}\big((1-\alpha\beta)(u_1+\alpha\beta u_3)\big),\qquad
f_2({\bf u})=\wp_{E_2}\big((\alpha\beta+1)(u_1-\alpha\beta u_3)\big),
\end{gather*}
where ${\bf u}={}^t(u_1,u_3)$.

\begin{Lemma}[{\cite{A-E-E-2004}, \cite[p.~39]{Baker-1907}, \cite[pp.~228, 229]{BEL-2012}}]
The following relations hold:
\begin{gather}
\wp_{1,1,1}^2=4\wp_{3,3}+4\lambda_4\wp_{1,1}+4\wp_{1,1}^3+4\wp_{1,3}\wp_{1,1}
+4\lambda_2\wp_{1,1}^2+4\lambda_6,\label{5.22.1}
\\
\wp_{1,1,1}\wp_{1,1,3}=2\lambda_8+2\wp_{1,3}^2-2\wp_{3,3}\wp_{1,1}+2\lambda_4\wp_{1,3}
+4\wp_{1,3}\wp_{1,1}^2+4\lambda_2\wp_{1,3}\wp_{1,1},\label{5.22.2}
\\
\wp_{1,1,3}^2=-4\wp_{3,3}\wp_{1,3}+4\lambda_2\wp_{1,3}^2+4\wp_{1,1}\wp_{1,3}^2,\label{5.22.3}
\\
\wp_{3,3}=-\wp_{1,1}\wp_{1,3}+\frac{1}{4}\wp_{1,1,1}^2-\wp_{1,1}^3-\lambda_2\wp_{1,1}^2
-\lambda_4\wp_{1,1}-\lambda_6,\label{3.23.111}
\\
\wp_{1,3,3}=\wp_{1,1,1}\wp_{1,3}-\wp_{1,1}\wp_{1,1,3},\label{3.23.112}
\\
\wp_{3,3,3}=2\wp_{1,1}\wp_{1,3,3}-\wp_{3,3}\wp_{1,1,1}-\wp_{1,3}\wp_{1,1,3}
+2\lambda_2\wp_{1,3,3}-\lambda_4\wp_{1,1,3}. \label{3.23.113}
\end{gather}
\end{Lemma}

\begin{Theorem}\label{3.6.2}
The hyperelliptic function $f_1$ is expressed in terms of $\wp_{1,1}$, $\wp_{1,3}$, and $\wp_{3,3}$ as follows:
\begin{gather}
f_1({\bf u})=\frac{-\alpha\beta\wp_{1,3}^2-A_1\wp_{1,3}-\alpha\beta(\wp_{1,1}-\alpha\beta) \big(\wp_{3,3}-\alpha^3\beta^3\big)}{(\alpha\beta-1)^2\big(\wp_{1,3}+\alpha^2\beta^2\big)^2},\label{4.13.112}
\end{gather}
where ${\bf u}\in\mathbb{C}^2$ and
\begin{gather*}
A_1=\wp_{3,3}+\alpha^2\beta^2\wp_{1,1}+\alpha\beta(\alpha\beta-1)^2(\alpha-\beta)^2-2\alpha^3\beta^3.
\end{gather*}
Here for simplicity we denote $\wp_{j,k}({\bf u})$ by $\wp_{j,k}$.
\end{Theorem}

\begin{proof}
Let $U$ be the subset of $\mbox{Sym}^2(V)$ consisting of elements $(P,Q)$ such that $P,Q\in V\backslash\{\infty\}$, $x_1,x_2\neq\alpha\beta$, $x_1\neq x_2$, and $x_1x_2\neq\alpha^2\beta^2$,
where $P=(x_1,y_1)$ and $Q=(x_2,y_2)$. Then $U$ is an open set in $\mbox{Sym}^2(V)$. Let $U'=I(U)$.
Since the restriction of $I$ to $U$ is injective, $U'$ is an open set in $\operatorname{Jac}(V)$. We~take a point ${\bf u}\in U'$. Then there exists $(P,Q)\in U$ such that ${\bf u}=I\left((P,Q)\right)$. We~have
\begin{gather*}
{\bf u}=\int_{\infty}^{P}{\boldsymbol \omega}+\int_{\infty}^{Q}{\boldsymbol \omega}=\int_{O_1}^{P}{\boldsymbol \omega}+\int_{O_1}^{Q}{\boldsymbol \omega}-2\int_{O_1}^{\infty}{\boldsymbol \omega}.
\end{gather*}
Let $P=(x_1,y_1)$ and $Q=(x_2,y_2)$. We~have $\varphi_1(\infty)=(0,0)$, which is a branch point of $E_1$.
Since the map $\varphi_{1,*}$ is the group homomorphism, we have
\begin{align*}
\varphi_{1,*}({\bf u})&=\varphi_{1,*}\bigg(\int_{O_1}^{P}{\boldsymbol \omega}\bigg) +\varphi_{1,*}\bigg(\int_{O_1}^{Q}{\boldsymbol \omega}\bigg) -2\varphi_{1,*}\bigg(\int_{O_1}^{\infty}{\boldsymbol \omega}\bigg)
\\
&=\varphi_1(P)+\varphi_1(Q)-2\varphi_1(\infty)=\varphi_1(P)+\varphi_1(Q).
\end{align*}
Let $\varphi_1(P)=(X_1,Y_1)$ and $\varphi_1(Q)=(X_2,Y_2)$.
Then we have
\begin{gather}
X_i=\frac{(\alpha-\beta)^2x_i}{(x_i-\alpha\beta)^2},\qquad Y_i=\frac{(\alpha-\beta)^2y_i}{(\alpha\beta-1)(x_i-\alpha\beta)^3},\qquad i=1,2.\label{5.13.1}
\end{gather}
From $x_1\neq x_2$ and $x_1x_2\neq\alpha^2\beta^2$, we have $X_1\neq X_2$.
Let $\varphi_1(P)+\varphi_1(Q)=(X_3,Y_3)$. Then, from the definition of the addition defined in $E_1$ (cf.~\cite[Chapter~III, Section~2]{silverman}), we have
\begin{gather}
X_3=\left(\frac{Y_2-Y_1}{X_2-X_1}\right)^2-X_1-X_2+1+\frac{(\alpha-\beta)^2}{(\alpha\beta-1)^2}.\label{4.13.5}
\end{gather}
From the well-known solution of the Jacobi inversion problem, we have
\begin{gather}
x_1+x_2=\wp_{1,1}({\bf u}),\qquad x_1x_2=-\wp_{1,3}({\bf u}),\label{5.13.2}
\\
2y_i=-\wp_{1,1,1}({\bf u})x_i-\wp_{1,1,3}({\bf u}),\qquad i=1,2.\label{5.13.3}
\end{gather}
From (\ref{5.13.1}), (\ref{5.13.2}), and (\ref{5.13.3}), we have
\begin{gather}
\frac{Y_2-Y_1}{X_2-X_1}=B({\bf u}),\qquad X_1+X_2=C({\bf u}),\label{4.13.3}
\end{gather}
where
\begin{gather*}
B({\bf u})=\frac{\big(\wp_{1,1}\wp_{1,3}-3\alpha\beta\wp_{1,3}-\alpha^3\beta^3\big)\wp_{1,1,1}
-\big(\wp_{1,1}^2+\wp_{1,3}-3\alpha\beta\wp_{1,1}+3\alpha^2\beta^2\big)
\wp_{1,1,3}}{2(\alpha\beta-1)\big({-}\alpha^2\beta^2+\alpha\beta\wp_{1,1}
+\wp_{1,3}\big)\big(\alpha^2\beta^2+\wp_{1,3}\big)},
\\
C({\bf u})=\frac{(\alpha-\beta)^2\big(\alpha^2\beta^2\wp_{1,1}+4\alpha\beta\wp_{1,3} -\wp_{1,1}\wp_{1,3}\big)}{\big(\alpha^2\beta^2-\alpha\beta\wp_{1,1}-\wp_{1,3}\big)^2},
\end{gather*}
where for simplicity we denote $\wp_{j,k}({\bf u})$ and $\wp_{j,k,\ell}({\bf u})$ by $\wp_{j,k}$ and $\wp_{j,k,\ell}$, respectively.
From (\ref{4.13.5}) and (\ref{4.13.3}), for ${\bf u}\in U'$, we obtain
\begin{gather*}
f_1({\bf u})=\wp_{E_1}(\psi_{1,*}({\bf u}))=X_3=B({\bf u})^2-C({\bf u})+1+\frac{(\alpha-\beta)^2}{(\alpha\beta-1)^2}.
\end{gather*}
We replace $\wp_{1,1,1}^2$, $\wp_{1,1,1}\wp_{1,1,3}$, and $\wp_{1,1,3}^2$ in the above expression of $f_1({\bf u})$ with the right hand sides of (\ref{5.22.1}), (\ref{5.22.2}), and (\ref{5.22.3}), respectively.
Then we find that the relation (\ref{4.13.112}) holds on $U'$.
When we regard $f_1$ as a meromorphic function on $\mathbb{C}^2$, the relation (\ref{4.13.112}) holds on $\pi^{-1}(U')$.
Since~$\pi^{-1}(U')$ is an open set in~$\mathbb{C}^2$, the relation (\ref{4.13.112}) holds on $\mathbb{C}^2$.
\end{proof}

\begin{Corollary}\label{3.12.111}
The hyperelliptic function $f_2$ is expressed in terms of $\wp_{1,1}$, $\wp_{1,3}$, and $\wp_{3,3}$ as follows:
\begin{gather}
f_2({\bf u})=\frac{\alpha\beta\wp_{1,3}^2-A_2\wp_{1,3}+\alpha\beta(\wp_{1,1}+\alpha\beta)\big(\wp_{3,3}+\alpha^3\beta^3\big)} {(\alpha\beta+1)^2\big(\wp_{1,3}+\alpha^2\beta^2\big)^2},\label{4.13.333}
\end{gather}
where ${\bf u}\in\mathbb{C}^2$ and
\begin{gather*}
A_2=\wp_{3,3}+\alpha^2\beta^2\wp_{1,1}-\alpha\beta(\alpha\beta+1)^2(\alpha+\beta)^2+2\alpha^3\beta^3.
\end{gather*}
Here for simplicity we denote $\wp_{j,k}({\bf u})$ by $\wp_{j,k}$.
\end{Corollary}

\begin{proof}In the same way as the proof of Proposition~\ref{3.6.1}, we can obtain the expression (\ref{4.13.333}) of~$f_2$ by changing $\beta$ with $-\beta$ in the expression of $f_1$ given in (\ref{4.13.112}).
\end{proof}

\section[The map from Jac(E\_i) to Jac(V)]{The map from $\boldsymbol{\operatorname{Jac}(E_i)}$ to $\boldsymbol{\operatorname{Jac}(V)}$}

Since the degree of $\varphi_i$ is 2 for $i=1,2$, the preimage $\varphi_i^{-1}(S)$ consists of two elements with multiplicity for any $S\in E_i$. We~consider the morphisms
\begin{gather*}
\chi\colon\quad V\to V,\qquad (x,y)\mapsto(x,-y),
\\
\chi_i\colon\quad E_i\to E_i,\qquad (X,Y)\mapsto(X,-Y),\qquad i=1,2.
\end{gather*}
We consider the maps
\begin{gather*}
\varphi_1^*\colon\quad E_1\to\operatorname{Jac}(V),\qquad
S\to \int_{O_1}^{S_1}{\boldsymbol \omega}+\int_{\chi(O_1)}^{S_2}{\boldsymbol \omega},
\end{gather*}
where $\varphi_1^{-1}(S)=\{S_1,S_2\}$, and
\begin{gather*}
\varphi_2^*\colon\quad E_2\to\operatorname{Jac}(V),\qquad
T\to \int_{O_2}^{T_1}{\boldsymbol \omega}+\int_{\chi(O_2)}^{T_2}{\boldsymbol \omega},
\end{gather*}
where $\varphi_2^{-1}(T)=\{T_1,T_2\}$.

\begin{Lemma}\label{3.8.111}
We have $\varphi_1^{-1}(\infty)=\{O_1, \chi(O_1)\}$ and $\varphi_1^{-1}((0,0))=\{(0,0), \infty\}$.
For $S=(X,Y)\in E_1\backslash\{\infty, (0,0)\}$, let $\varphi_1^{-1}(S)=\{S_1,S_2\}$ and $S_i=(x_i,y_i)$ for $i=1,2$.
Then we have
\begin{gather*}
x_1=\alpha\beta+\frac{(\alpha-\beta)\big(\alpha-\beta+\sqrt{4\alpha\beta X+(\alpha-\beta)^2}\big)}{2X},
\\
y_1=\frac{(\alpha\beta-1)(\alpha-\beta)Y\big(\alpha-\beta+\sqrt{4\alpha\beta X+(\alpha-\beta)^2}\big)^3}{8X^3},
\\
x_2=\alpha\beta+\frac{(\alpha-\beta)\big(\alpha-\beta-\sqrt{4\alpha\beta X+(\alpha-\beta)^2}\big)}{2X},
\\
y_2=\frac{(\alpha\beta-1)(\alpha-\beta)Y\big(\alpha-\beta-\sqrt{4\alpha\beta X+(\alpha-\beta)^2}\big)^3}{8X^3}.
\end{gather*}
\end{Lemma}

\begin{proof}
From the definition of $\varphi_1$ and the direct calculations, we obtain the statement of the lemma.
\end{proof}

\begin{Lemma}\label{5.14.1}
We have $\varphi_2^{-1}(\infty)=\{O_2, \chi(O_2)\}$ and $\varphi_2^{-1}((0,0))=\{(0,0), \infty\}$.
For $T=(X,Y)\in E_2\backslash\{\infty, (0,0)\}$, let $\varphi_2^{-1}(T)=\{T_1,T_2\}$ and $T_i=(x_i,y_i)$ for $i=1,2$.
Then we have
\begin{gather*}
x_1=-\alpha\beta+\frac{(\alpha+\beta)\big(\alpha+\beta+\sqrt{-4\alpha\beta X+(\alpha+\beta)^2}\big)}{2X},
\\
y_1=\frac{-(\alpha\beta+1)(\alpha+\beta)Y\big(\alpha+\beta+\sqrt{-4\alpha\beta X+(\alpha+\beta)^2}\big)^3}{8X^3},
\\
x_2=-\alpha\beta+\frac{(\alpha+\beta)\big(\alpha+\beta-\sqrt{-4\alpha\beta X+(\alpha+\beta)^2}\big)}{2X},
\\
y_2=\frac{-(\alpha\beta+1)(\alpha+\beta)Y\big(\alpha+\beta-\sqrt{-4\alpha\beta X+(\alpha+\beta)^2}\big)^3}{8X^3}.
\end{gather*}
\end{Lemma}

\begin{proof}
From the definition of $\varphi_2$ and the direct calculations, we obtain the statement of the lemma.
\end{proof}

\begin{Lemma}\label{3.10.1}
For $i=1,2$, the map $\varphi_i^*$ is the holomorphic map between the complex manifolds~$E_i$ and $\operatorname{Jac}(V)$, and the map $\varphi_i^*$ is the homomorphism
between the groups $E_i$ and $\operatorname{Jac}(V)$.
\end{Lemma}

\begin{proof}
We consider the following point on $E_1$
\begin{gather*}
D=\bigg({-}\frac{(\alpha-\beta)^2}{4\alpha\beta},
\frac{(\alpha-\beta)^2(\alpha+\beta)(\alpha\beta+1)}{8\alpha\beta\sqrt{-\alpha\beta}(\alpha\beta-1)}\bigg).
\end{gather*}
Let $E_1'=E_1\backslash\{(0,0), D, \chi_1(D), \infty\}$.
From the definition of $\varphi_1^*$ and Lemma \ref{3.8.111}, the map $\varphi_1^*$ is holomorphic on $E_1'$.
Since the map $\varphi_1^*$ is bounded around the points $(0,0), D, \chi_1(D)$, and $\infty$, the map $\varphi_1^*$ is holomorphic on $E_1$.
Since $\varphi_1^{-1}(\infty)=\{O_1, \chi(O_1)\}$, we have $\varphi_1^*(\infty)={\bf 0}$, where ${\bf 0}$ is the unit element of $\operatorname{Jac}(V)$.
Therefore the map $\varphi_1^*$ is the homomorphism between the groups~$E_1$ and $\operatorname{Jac}(V)$.
By changing $\beta$ with $-\beta$, we can obtain the statement of the lemma for $\varphi_2^*$.
\end{proof}

We consider the maps
\begin{gather*}
\varphi^*\colon\quad E_1\times E_2\to\operatorname{Jac}(V),\qquad (S,T)\mapsto \varphi_1^*(S)+\varphi_2^*(T),
\\
\varphi_*\colon\quad \operatorname{Jac}(V)\to E_1\times E_2,\qquad {\bf u}\mapsto(\varphi_{1,*}({\bf u}), \varphi_{2,*}({\bf u})).
\end{gather*}

\begin{Remark}
It is known that the map $\varphi_*$ is a $(2,2)$-isogeny and $\varphi^*$ is its dual isogeny (cf.~\cite[Proposition~4.8]{B-W-2003}, \cite[Section~1.3.5, Theorem~2.40]{BCMS2021}, \cite[Lemma~2.6]{BD}, \cite[Section~4]{F2}, \cite[Section~1]{FK}, \cite[Section~2]{Ku},
\cite[Theorem~4.2]{W}).
\end{Remark}

\begin{Lemma} [{\cite[Theorem~2.40]{BCMS2021}, \cite[Lemma~2.6]{BD}, \cite[Section~4]{F2}, \cite[Theorem~4.2]{W}}]\label{3.9.1}
We have
\begin{gather*}
\varphi_*\circ\varphi^*(S,T)=(2S, 2T).
\end{gather*}
\end{Lemma}

\begin{proof}
For the sake to be complete and self-contained, we give a proof of this lemma.
For $(S,T)\in E_1\times E_2$, we have
\begin{align*}
\varphi_{1,*}\circ\varphi^*(S,T)&=\varphi_{1,*}\bigg(\int_{O_1}^{S_1}{\boldsymbol \omega}+\int_{\chi(O_1)}^{S_2}{\boldsymbol \omega}+\int_{O_2}^{T_1}{\boldsymbol \omega}+\int_{\chi(O_2)}^{T_2}{\boldsymbol \omega}\bigg)
\\
&=\varphi_{1,*}\bigg(\int_{O_1}^{S_1}\!{\boldsymbol \omega}+\!\int_{O_1}^{S_2}\!{\boldsymbol \omega}-\!\int_{O_1}^{\chi(O_1)}\!{\boldsymbol \omega}+\!\int_{O_1}^{T_1}\!{\boldsymbol \omega}-\!\int_{O_1}^{O_2}\!{\boldsymbol \omega}+\!\int_{O_1}^{T_2}\!{\boldsymbol \omega}
-\!\int_{O_1}^{\chi(O_2)}\!{\boldsymbol \omega}\bigg)
\\
&=\varphi_1(S_1)+\varphi_1(S_2)-\varphi_1(\chi(O_1))+\varphi_1(T_1)-\varphi_1(O_2) +\varphi_1(T_2)-\varphi_1(\chi(O_2))
\\
&=2S+\varphi_1(T_1)+\varphi_1(T_2)-\varphi_1(O_2)-\varphi_1(\chi(O_2)),
\end{align*}
where $\varphi_1^{-1}(S)=\{S_1,S_2\}$ and $\varphi_2^{-1}(T)=\{T_1,T_2\}$. From (\ref{3.16.11}), Lemma \ref{5.14.1}, and the direct calculations, we can check $\varphi_1(T_2)=\chi_1(\varphi_1(T_1))$ and $\varphi_1(\chi(O_2))=\chi_1(\varphi_1(O_2))$.
Therefore we obtain
\begin{gather*}
\varphi_{1,*}\circ\varphi^*(S,T)=2S.
\end{gather*}
By changing $\beta$ with $-\beta$, we can find
\begin{gather*}
\varphi_{2,*}\circ\varphi^*(S,T)=2T.
\end{gather*}
Therefore we obtain the statement of the lemma.
\end{proof}

We consider the maps
\begin{gather*}
\psi_i^*=\varphi_i^*\circ J_i^{-1}\colon\quad \operatorname{Jac}(E_i)\to\operatorname{Jac}(V),\qquad i=1,2.
\end{gather*}
For $i=1,2$, the map $\psi_i^*$ is the holomorphic map between the complex manifolds $\operatorname{Jac}(E_i)$ and $\operatorname{Jac}(V)$, and the map $\psi_i^*$ is the homomorphism between the groups $\operatorname{Jac}(E_i)$ and $\operatorname{Jac}(V)$.
We~consider the maps
\begin{gather*}
\psi^*\colon\quad \operatorname{Jac}(E_1)\times\operatorname{Jac}(E_2)\to\operatorname{Jac}(V),\qquad {}^t(v_1,v_2)\mapsto \psi_1^*(v_1)+\psi_2^*(v_2),
\\
\psi_*\colon\quad \operatorname{Jac}(V)\to\operatorname{Jac}(E_1)\times\operatorname{Jac}(E_2),\qquad {\bf u}\mapsto {}^t(\psi_{1,*}({\bf u}), \psi_{2,*}({\bf u})).
\end{gather*}

\begin{Proposition}\label{2.15.1}
For $i=1,2$, we consider the maps
\begin{gather*}
\widetilde{\psi}_i^*\colon\quad \mathbb{C}\to\mathbb{C}^2,\qquad v\mapsto {}^t\left(c_iv, d_iv\right),
\end{gather*}
where
\begin{gather}
c_1=\frac{1}{1-\alpha\beta},\qquad d_1=\frac{1}{\alpha\beta(1-\alpha\beta)},\qquad c_2=\frac{1}{1+\alpha\beta},\qquad d_2=-\frac{1}{\alpha\beta(1+\alpha\beta)}.\label{3.9.2}
\end{gather}
Then we have $\psi_i^*\circ\pi_i=\pi\circ\widetilde{\psi}_i^*$, i.e., we have the following commutative diagram:
\begin{gather*}
 \begin{CD}
 \mathbb{C} @>{\widetilde{\psi}_i^*}>> \mathbb{C}^2 \\
 @V{\pi_i}VV @V{\pi}VV \\
 \operatorname{Jac}(E_i) @>{\psi_i^*} >> \operatorname{Jac}(V).
 \end{CD}
\end{gather*}
\end{Proposition}

\begin{proof}
For $i=1,2$, since $\psi_i^*$ is holomorphic and the group homomorphism, it is well known that there exist complex numbers $c_i$ and $d_i$ such that
$\psi_i^*\circ\pi_i=\pi\circ\widetilde{\psi}_i^*$, where
\begin{gather*}
\widetilde{\psi}_i^*\colon\quad \mathbb{C}\to\mathbb{C}^2,\qquad v\mapsto {}^t\left(c_iv, d_iv\right),
\end{gather*}
(e.g., \cite[Proposition~1.2.1]{B-H-2004}). We~consider the maps
\begin{gather*}
\widetilde{\psi}^*\colon\quad \mathbb{C}^2\to\mathbb{C}^2,\qquad \begin{pmatrix}v_1\\v_2\end{pmatrix}\mapsto\begin{pmatrix}c_1&c_2\\d_1&d_2\end{pmatrix}
\begin{pmatrix}v_1\\v_2\end{pmatrix}\!,
\\
\widetilde{\psi}_*\colon\quad \mathbb{C}^2\to\mathbb{C}^2,\qquad \begin{pmatrix}u_1\\u_3\end{pmatrix}\mapsto\begin{pmatrix}a_1&b_1\\a_2&b_2\end{pmatrix}
\begin{pmatrix}u_1\\u_3\end{pmatrix}\!,
\end{gather*}
where $a_1$, $b_1$, $a_2$, and $b_2$ are defined in Proposition~\ref{3.6.1}. We~consider the map
\begin{gather*}
\pi_{1,2}\colon\quad \mathbb{C}^2\to\operatorname{Jac}(E_1)\times\operatorname{Jac}(E_2),\qquad {}^t(v_1,v_2)\mapsto{}^t(\pi_1(v_1), \pi_2(v_2)).
\end{gather*}
Then we have the following commutative diagrams:
\begin{gather*}
 \begin{CD}
 \mathbb{C}^2 @>{\widetilde{\psi}^*}>> \mathbb{C}^2 \\
 @V{\pi_{1,2}}VV @V{\pi}VV \\
 \operatorname{Jac}(E_1)\times\operatorname{Jac}(E_2) @>{\psi^*} >> \operatorname{Jac}(V),
 \end{CD}\qquad
 \begin{CD}
 \mathbb{C}^2 @>{\widetilde{\psi}_*}>> \mathbb{C}^2 \\
 @V{\pi}VV @V{\pi_{1,2}}VV \\
 \operatorname{Jac}(V) @>{\psi_*} >>\operatorname{Jac}(E_1)\times\operatorname{Jac}(E_2).
 \end{CD}
\end{gather*}
Therefore we obtain the following commutative diagram:
\begin{gather*}
 \begin{CD}
 \mathbb{C}^2 @>{\widetilde{\psi}_*\circ\widetilde{\psi}^*}>> \mathbb{C}^2 \\
 @V{\pi_{1,2}}VV @V{\pi_{1,2}}VV \\
 \operatorname{Jac}(E_1)\times\operatorname{Jac}(E_2) @>{\psi_*\circ\psi^*} >> \operatorname{Jac}(E_1)\times\operatorname{Jac}(E_2).
 \end{CD}
\end{gather*}
We have
\begin{gather*}
\widetilde{\psi}_*\circ\widetilde{\psi}^*\colon\quad \mathbb{C}^2\to\mathbb{C}^2,\qquad
\begin{pmatrix}v_1\\v_2\end{pmatrix}
\mapsto\begin{pmatrix}a_1&b_1\\a_2&b_2\end{pmatrix}
\begin{pmatrix}c_1&c_2\\d_1&d_2\end{pmatrix}\begin{pmatrix}v_1\\v_2\end{pmatrix}\!.
\end{gather*}
From Lemma \ref{3.9.1}, we obtain
\begin{gather*}
\begin{pmatrix}a_1&b_1\\a_2&b_2\end{pmatrix}
\begin{pmatrix}c_1&c_2\\d_1&d_2\end{pmatrix}=\begin{pmatrix}2&0\\0&2\end{pmatrix}\!.
\end{gather*}
From (\ref{3.8.123}), we obtain (\ref{3.9.2}).
\end{proof}

For $i=1,2$, let ${\bf k}_i={}^t(c_i, d_i)$ and $K=({\bf k}_1, {\bf k}_2)=\left(\begin{smallmatrix}c_1&c_2\\d_1&d_2\end{smallmatrix}\right)$, where $c_1$, $d _1$, $c_2$, and $d_2$ are given in (\ref{3.9.2}).

\begin{Theorem}\label{2.17.1}
For $v\in\mathbb{C}$, the following relations hold:
\begin{gather}
\wp_{1,1}({\bf k}_1v)=2\alpha\beta+\frac{(\alpha-\beta)^2}{\wp_{E_1}(v)},\qquad \wp_{1,3}({\bf k}_1v)=-\alpha^2\beta^2,
\label{3.23.1}
\\
\wp_{3,3}({\bf k}_1v)=\alpha^2\beta^2\big\{(\alpha\beta-1)^2\wp_{E_1}(v)+2\alpha\beta\big\},
\label{3.23.2}
\\
\wp_{1,1,1}({\bf k}_1v)=\frac{(\alpha\beta-1)\wp_{E_1}'(v) \big\{(\alpha-\beta)^2+\alpha\beta\wp_{E_1}(v)\big\}} {\wp_{E_1}(v)^2},
\label{3.23.3}
\\
\wp_{1,1,3}({\bf k}_1v)=-\frac{\alpha^2\beta^2(\alpha\beta-1)\wp_{E_1}'(v)}{\wp_{E_1}(v)},
\label{3.23.4}
\\
\wp_{1,3,3}({\bf k}_1v)=\frac{\alpha^3\beta^3(\alpha\beta-1)\wp_{E_1}'(v)}{\wp_{E_1}(v)},
\label{3.23.5}
\\
\wp_{3,3,3}({\bf k}_1v)=-\frac{\alpha^3\beta^3(\alpha\beta-1) \big\{(\alpha\beta-1)^2\wp_{E_1}(v)+\alpha\beta\big\}\wp_{E_1}'(v)}{\wp_{E_1}(v)}.
\label{3.23.6}
\end{gather}
\end{Theorem}

\begin{proof}We consider the set $E_1'$ defined in the proof of Lemma \ref{3.10.1} and $U''=J_1(E_1')$. Then $U''$ is an open set in $\operatorname{Jac}(E_1)$. We~take a point $S=(X,Y)\in E_1'$ and set $\overline{v}=J_1(S)$. We~have
\begin{gather*}
\varphi_1^*(S)=\int_{O_1}^{S_1}{\boldsymbol \omega}+\int_{\chi(O_1)}^{S_2}{\boldsymbol \omega}=\int_{\infty}^{S_1}{\boldsymbol \omega}+\int_{\infty}^{S_2}{\boldsymbol \omega},
\end{gather*}
where $\varphi_1^{-1}(S)=\{S_1,S_2\}$.
From Proposition \ref{2.15.1}, we have $\varphi_1^*(S)=\pi({\bf k}_1v)$,
where $v\in\mathbb{C}$ such that $\pi_1(v)=\overline{v}$.
From Lemma \ref{3.8.111} and $S\neq(0,0)$, we have $S_1, S_2\neq\infty$.
Let $S_i=(x_i,y_i)$ for $i=1,2$.
From Lemma \ref{3.8.111} and $S\in E_1'$, we have $x_1\neq x_2$, i.e., $S_2\neq\chi(S_1)$.
Therefore, from the well-known solution of the Jacobi inversion problem, we have
\begin{gather*}
\wp_{1,1}({\bf k}_1v)=x_1+x_2,\qquad \wp_{1,3}({\bf k}_1v)=-x_1x_2,
\\
2y_i=-\wp_{1,1,1}({\bf k}_1v)x_i-\wp_{1,1,3}({\bf k}_1v),\qquad i=1,2.
\end{gather*}
Thus we have
\begin{gather*}
\wp_{1,1,1}({\bf k}_1v)=-2\frac{y_1-y_2}{x_1-x_2},\qquad \wp_{1,1,3}({\bf k}_1v)=2\frac{x_2y_1-x_1y_2}{x_1-x_2}.
\end{gather*}
We have $X=\wp_{E_1}(v)$ and $Y=-\wp_{E_1}'(v)/2$.
Therefore, from Lemma \ref{3.8.111} and the direct calculations, we find that (\ref{3.23.1}), (\ref{3.23.3}), and (\ref{3.23.4}) hold for $v\in\pi_1^{-1}(U'')$.
Since $\pi_1^{-1}(U'')$ is an open set in $\mathbb{C}$, (\ref{3.23.1}), (\ref{3.23.3}), and (\ref{3.23.4}) hold on $\mathbb{C}$.
From (\ref{3.23.111}), (\ref{3.23.112}), (\ref{3.23.113}), and the direct calculations, we obtain (\ref{3.23.2}), (\ref{3.23.5}), and (\ref{3.23.6}).
\end{proof}

\begin{Corollary}\label{3.12.1111}
For $v\in\mathbb{C}$, the following relations hold:
\begin{gather}
\wp_{1,1}({\bf k}_2v)=-2\alpha\beta+\frac{(\alpha+\beta)^2}{\wp_{E_2}(v)},\qquad
\wp_{1,3}({\bf k}_2v)=-\alpha^2\beta^2,
\label{3.24.1}
\\
\wp_{3,3}({\bf k}_2v)=\alpha^2\beta^2\big\{(\alpha\beta+1)^2\wp_{E_2}(v)-2\alpha\beta\big\},
\label{3.24.2}
\\
\wp_{1,1,1}({\bf k}_2v)=-\frac{(\alpha\beta+1)\wp_{E_2}'(v)\big\{(\alpha+\beta)^2-\alpha\beta\wp_{E_2}(v)\big\}} {\wp_{E_2}(v)^2},
\label{3.24.3}
\\
\wp_{1,1,3}({\bf k}_2v)=\frac{\alpha^2\beta^2(\alpha\beta+1)\wp_{E_2}'(v)}{\wp_{E_2}(v)},
\label{3.24.4}
\\
\wp_{1,3,3}({\bf k}_2v)=\frac{\alpha^3\beta^3(\alpha\beta+1)\wp_{E_2}'(v)}{\wp_{E_2}(v)},
\label{3.24.5}
\\
\wp_{3,3,3}({\bf k}_2v)=-\frac{\alpha^3\beta^3(\alpha\beta+1) \big\{(\alpha\beta+1)^2\wp_{E_2}(v)-\alpha\beta\big\}\wp_{E_2}'(v)}{\wp_{E_2}(v)}.
\label{3.24.6}
\end{gather}
\end{Corollary}

\begin{proof}
In the same way as the proof of Proposition \ref{3.6.1}, we can obtain the expressions (\ref{3.24.1})--(\ref{3.24.6}) by changing $\beta$ with $-\beta$ in the expressions (\ref{3.23.1})--(\ref{3.23.6}), respectively.
\end{proof}

\begin{Theorem}\label{3.25.1}For ${\bf v}={}^t(v_1,v_2)\in\mathbb{C}^2$, we have
\begin{gather*}
\wp_{1,1}(K{\bf v})=-\frac{p_2 \big\{(\alpha\beta-1)^2(\alpha+\beta)^2\wp_{E_1}^2 +(\alpha\beta+1)^2(\alpha-\beta)^2\wp_{E_2}^2+p_3\big\}}{2p_1^2},
\\
\wp_{1,3}(K{\bf v})=\alpha^2\beta^2+\frac{\alpha^2\beta^2 p_2 \big\{\big(\alpha^2\beta^2-1\big)\wp_{E_1}'\wp_{E_2}'-p_2\big\}}{2p_1^2},
\\
\wp_{3,3}(K{\bf v})=-\frac{\alpha^2\beta^2 p_2 \big\{\big(\alpha^2\beta^2-1\big)^2\wp_{E_1}^2\wp_{E_2}^2+\big(\alpha^2-\beta^2\big)^2+p_3\big\}}{2p_1^2},
\end{gather*}
where
\begin{gather*}
p_1=\big\{(\alpha\beta-1)^2\wp_{E_1}-(\alpha\beta+1)^2\wp_{E_2}+8\alpha\beta\big\} \wp_{E_1}\wp_{E_2}-(\alpha+\beta)^2\wp_{E_1}+(\alpha-\beta)^2\wp_{E_2},
\\
p_2=\big(\alpha^2\beta^2-1\big)\wp_{E_1}'\wp_{E_2}'-2\big\{(\alpha\beta-1)^2 \wp_{E_1}+(\alpha\beta+1)^2\wp_{E_2}-2\big(\alpha^2+1\big)\big(\beta^2+1\big)\big\}\wp_{E_1}\wp_{E_2}
\\ \hphantom{p_2=}
{} -2\big\{(\alpha+\beta)^2\wp_{E_1}+(\alpha-\beta)^2\wp_{E_2}\big\},
\\
p_3=-2\alpha\beta p_1-2\big(\alpha^2\beta^4+\alpha^4\beta^2-4\alpha^2\beta^2+\alpha^2+\beta^2\big) \wp_{E_1}\wp_{E_2}.
\end{gather*}
Here for simplicity we denote $\wp_{E_1}(v_1)$ and $\wp_{E_2}(v_2)$ by $\wp_{E_1}$ and $\wp_{E_2}$.
\end{Theorem}

\begin{proof}
In~\cite[Theorem~3.3]{G}, for ${\bf u}, {\bf v}\in\mathbb{C}^2$, the following functions are introduced:
\begin{gather*}
q({\bf u},{\bf v})=\wp_{3,3}({\bf u})-\wp_{3,3}({\bf v})+\wp_{1,3}({\bf u})\wp_{1,1}({\bf v}) -\wp_{1,3}({\bf v})\wp_{1,1}({\bf u}),
\\
q_1({\bf u},{\bf v})=\wp_{1,3,3}({\bf u})-\wp_{1,3,3}({\bf v})+\wp_{1,1,3}({\bf u})\wp_{1,1}({\bf v}) -\wp_{1,1,3}({\bf v})\wp_{1,1}({\bf u})+\wp_{1,1,1}({\bf v})\wp_{1,3}({\bf u})
\\
\hphantom{q_1({\bf u},{\bf v})=}
{}-\wp_{1,1,1}({\bf u})\wp_{1,3}({\bf v}),
\\
q_3({\bf u},{\bf v})=\wp_{3,3,3}({\bf u})-\wp_{3,3,3}({\bf v})+\wp_{1,3,3}({\bf u})\wp_{1,1}({\bf v}) -\wp_{1,3,3}({\bf v})\wp_{1,1}({\bf u})+\wp_{1,1,3}({\bf v})\wp_{1,3}({\bf u})
\\ \hphantom{q_3({\bf u},{\bf v})=}
{}-\wp_{1,1,3}({\bf u})\wp_{1,3}({\bf v}),
\\
q_{1,1}({\bf u},{\bf v})=8\lambda_2(\wp_{1,3}({\bf u})\wp_{1,1}({\bf v})-\wp_{1,3}({\bf v})\wp_{1,1}({\bf u}))+4\lambda_4(\wp_{1,3}({\bf u})-\wp_{1,3}({\bf v}))
\\ \hphantom{q_{1,1}({\bf u},{\bf v})=}
{}-4(\wp({\bf u})-\wp({\bf v}))-8(\wp_{3,3}({\bf u})\wp_{1,1}({\bf v})
-\wp_{3,3}({\bf v})\wp_{1,1}({\bf u}))+2(\wp_{1,1,3}({\bf u})\wp_{1,1,1}({\bf v})
\\ \hphantom{q_{1,1}({\bf u},{\bf v})=}
{}-\wp_{1,1,3}({\bf v})\wp_{1,1,1}({\bf u})),
\\
q_{1,3}({\bf u},{\bf v})=4\lambda_6(\wp_{1,3}({\bf u})-\wp_{1,3}({\bf v}))+2\lambda_4(\wp_{1,3}({\bf u})\wp_{1,1}({\bf v})-\wp_{1,3}({\bf v})\wp_{1,1}({\bf u}))
\\ \hphantom{q_{1,3}({\bf u},{\bf v})=}
{}-4(\wp_{3,3}({\bf u})\wp_{1,3}({\bf v})-\wp_{3,3}({\bf v})\wp_{1,3}({\bf u}))+2(\wp({\bf u})\wp_{1,1}({\bf v})-\wp({\bf v})\wp_{1,1}({\bf u}))
\\ \hphantom{q_{1,3}({\bf u},{\bf v})=}
{}-2\lambda_8(\wp_{1,1}({\bf u})-\wp_{1,1}({\bf v}))+\wp_{1,1,1}({\bf v})\wp_{1,3,3}({\bf u})-\wp_{1,1,1}({\bf u})\wp_{1,3,3}({\bf v}),
\\
q_{3,3}({\bf u},{\bf v})=4\lambda_6q({\bf u},{\bf v})+4\lambda_8(\wp_{1,3}({\bf u})-\wp_{1,3}({\bf v}))+4(\wp({\bf u})\wp_{1,3}({\bf v})-\wp({\bf v})\wp_{1,3}({\bf u}))
\\ \hphantom{q_{3,3}({\bf u},{\bf v})=}
{}+2(\wp_{1,3,3}({\bf u})\wp_{1,1,3}({\bf v})-\wp_{1,3,3}({\bf v})\wp_{1,1,3}({\bf u})),
\end{gather*}
where we set $\wp=\wp_{1,1}\wp_{3,3}-\wp_{1,3}^2$.
From $K{\bf v}={\bf k}_1v_1+{\bf k}_2v_2$ and the addition formulae for $\wp_{j,k}$, which are given in \cite[Theorem 3.3]{G}, we obtain
\begin{gather*}
\wp_{1,1}(K{\bf v})=-\wp_{1,1}({\bf k}_1v_1)-\wp_{1,1}({\bf k}_2v_2)+\frac{1}{4}\bigg(\frac{q_1({\bf k}_1v_1, {\bf k}_2v_2)}{q({\bf k}_1v_1, {\bf k}_2v_2)}\bigg)^2-\frac{q_{1,1}({\bf k}_1v_1, {\bf k}_2v_2)}{4q({\bf k}_1v_1, {\bf k}_2v_2)},
\\
\wp_{1,3}(K{\bf v})=-\wp_{1,3}({\bf k}_1v_1)-\wp_{1,3}({\bf k}_2v_2)+\frac{q_1({\bf k}_1v_1, {\bf k}_2v_2)q_3({\bf k}_1v_1, {\bf k}_2v_2)}{4q({\bf k}_1v_1, {\bf k}_2v_2)^2}-\frac{q_{1,3}({\bf k}_1v_1, {\bf k}_2v_2)}{4q({\bf k}_1v_1, {\bf k}_2v_2)},
\\
\wp_{3,3}(K{\bf v})=-\wp_{3,3}({\bf k}_1v_1)-\wp_{3,3}({\bf k}_2v_2)+\frac{1}{4}\bigg(\frac{q_3({\bf k}_1v_1, {\bf k}_2v_2)}{q({\bf k}_1v_1, {\bf k}_2v_2)}\bigg)^2-\frac{q_{3,3}({\bf k}_1v_1, {\bf k}_2v_2)}{4q({\bf k}_1v_1, {\bf k}_2v_2)}.
\end{gather*}
From Theorem \ref{2.17.1} and Corollary \ref{3.12.1111}, we obtain the expressions of the theorem.
\end{proof}

{\sloppy\begin{Remark}\label{5.8.111}
From Theorem \ref{3.25.1}, we obtain the decompositions of the functions $\wp_{1,1}$, \mbox{$\wp_{1,3}-\alpha^2\beta^2$}, and $\wp_{3,3}$ into the products of the meromorphic functions.
\end{Remark}

}

\begin{Remark}\label{3.28}
For ${\bf u}={}^t(u_1,u_3)\in\mathbb{C}^2$, we set $F({\bf u})=2\wp_{1,1}({\bf u})+\frac{2}{3}\lambda_2$ and $G({\bf u})=2\wp_{1,3}({\bf u})$.
It~is known that the function $F$ is a solution of the KdV-hierarchy (see \cite[Theorem~2.25]{BE} and \cite[Theorem~3.6]{BEL-2012})
\begin{gather*}
\frac{\partial F}{\partial u_3}=\frac{1}{4}\frac{\partial^3F}{\partial u_1^3}-\frac{3}{2}F\frac{\partial F}{\partial u_1},
\\
0=\frac{1}{4}\frac{\partial^3F}{\partial u_1^2\partial u_3} -\bigg(F+\frac{\lambda_2}{3}\bigg)\frac{\partial F}{\partial u_3}
-\frac{1}{2}\frac{\partial F}{\partial u_1}G,\qquad
\frac{\partial F}{\partial u_3}=\frac{\partial G}{\partial u_1}.
\end{gather*}
For ${\bf v}={}^t(v_1,v_2)\in\mathbb{C}^2$, we set $\widetilde{F}({\bf v})=\wp_{1,1}(K{\bf v})$ and $\widetilde{G}({\bf v})=\wp_{1,3}(K{\bf v})$.
From Theorem \ref{3.25.1}, we can express $\widetilde{F}({\bf v})$ and $\widetilde{G}({\bf v})$ in terms of $\wp_{E_1}(v_1)$ and $\wp_{E_2}(v_2)$. We~have $F({\bf u})=2\widetilde{F}\big(K^{-1}{\bf u}\big)+\frac{2}{3}\lambda_2$ and $G({\bf u})=2\widetilde{G}\big(K^{-1}{\bf u}\big)$.
Therefore, we can obtain a solution of the KdV-hierarchy in terms of the functions constructed by the Weierstrass elliptic functions $\wp_{E_1}$ and $\wp_{E_2}$.
\end{Remark}

\begin{Remark}\label{2021.11.22.2222}
It is well known that the functions $\wp_{1,1}$, $\wp_{1,3}$, and $\wp_{3,3}$ are coordinates of the Kummer surface (e.g., \cite{BEL, BEL-2012, Hudson-1990}).
From Theorem~\ref{3.25.1}, we obtain the expressions of the coordinates of the Kummer surface in terms of the Weierstrass elliptic functions $\wp_{E_1}$ and $\wp_{E_2}$.
\end{Remark}

\section{Comparison of our results and the results of \cite{BE, BEL-2012, Enolskii-Salerno-1996}}\label{5.8.222}

Let us consider the curve~$V$ defined by (\ref{4.7.1}).
Let $\widetilde{E}_1$ and $\widetilde{E}_2$ be the elliptic curves defined by
\begin{gather*}
\widetilde{E}_1=\big\{\big(\widetilde{X},\widetilde{Y}\big)\in\mathbb{C}^2\mid \widetilde{Y}^2=\widetilde{X}\big(\widetilde{X}-1\big)\big(\widetilde{X}-1/\kappa_1^2\big)\big\},
\\
\widetilde{E}_2=\big\{\big(\widetilde{X},\widetilde{Y}\big)\in\mathbb{C}^2\mid \widetilde{Y}^2=\widetilde{X}\big(\widetilde{X}-1\big)\big(\widetilde{X}-1/\kappa_2^2\big)\big\}.
\end{gather*}
For $i=1,2,$ let $\{\widetilde{\mathfrak{a}}^{(i)},\widetilde{\mathfrak{b}}^{(i)}\}$ be a canonical basis in the one-dimensional homology group of the curve~$\widetilde{E}_i$.
For $i=1,2$, we define the periods
\begin{gather*}
2\widetilde{\Omega}_i'=\int_{\widetilde{\mathfrak{a}}^{(i)}}-\frac{{\rm d}\widetilde{X}}{2\widetilde{Y}},
\qquad 2\widetilde{\Omega}_i''=\int_{\widetilde{\mathfrak{b}}^{(i)}}-\frac{{\rm d}\widetilde{X}}{2\widetilde{Y}}.
\end{gather*}
We define the period lattice $\widetilde{\Lambda}_i=\big\{2\widetilde{\Omega}_i'm_1+2\widetilde{\Omega}_i''m_2\mid m_1,m_2\in\mathbb{Z}\big\}$ and consider the Jacobian variety $\operatorname{Jac}\big(\widetilde{E}_i\big)=\mathbb{C}/\widetilde{\Lambda}_i$.
For $i=1,2$, we take the canonical basis $\big\{\widetilde{\mathfrak{a}}^{(i)},\widetilde{\mathfrak{b}}^{(i)}\big\}$ in the one-dimensional homology group of the curve $\widetilde{E}_i$ such that the following relations hold in $\operatorname{Jac}\big(\widetilde{E}_i\big)$
\begin{gather*}
\int_{\infty}^{(1/\kappa_i^2,0)}\bigg({-}\frac{{\rm d}\widetilde{X}}{2\widetilde{Y}}\bigg) =\widetilde{\Omega}_i',\qquad
\int_{\infty}^{(0,0)}\bigg({-}\frac{{\rm d}\widetilde{X}}{2\widetilde{Y}}\bigg) =\widetilde{\Omega}_i'',\qquad
\int_{\infty}^{(1,0)}\bigg({-}\frac{{\rm d}\widetilde{X}}{2\widetilde{Y}}\bigg) =\widetilde{\Omega}_i'+\widetilde{\Omega}_i''.
\end{gather*}
For $i=1,2$, let $\widetilde{\tau}_i=\big(\widetilde{\Omega}_i'\big)^{-1}\widetilde{\Omega}_i''$. We~consider the Jacobi elliptic functions $\mbox{sn}(u,\widetilde{\tau}_i)$, $\mbox{cn}(u,\widetilde{\tau}_i)$, and $\mbox{dn}(u,\widetilde{\tau}_i)$ associated with $\widetilde{\tau}_i$, where the definitions of the Jacobi elliptic functions are written in~Section~\ref{4.27.11}.
For $i=1,2$, let $m_i$ be the modulus of the Jacobi elliptic functions associated with~$\widetilde{\tau}_i$. Then we have $m_i^2=\kappa_i^2$ (cf.~\cite{H}).
In~\cite{BE, BEL-2012, Enolskii-Salerno-1996}, the Jacobi elliptic functions with modulus $\kappa_i$ are considered.
Since the Jacobi elliptic functions are determined by the second power of the modulus, the Jacobi elliptic functions $\mbox{sn}(u,\widetilde{\tau}_i), \mbox{cn}(u,\widetilde{\tau}_i)$, and $\mbox{dn}(u,\widetilde{\tau}_i)$
are equal to the ones considered in \cite{BE, BEL-2012, Enolskii-Salerno-1996}.
For ${\bf u}={}^t(u_1,u_3)\in\mathbb{C}^2$, let
\begin{gather*}
w_1=\sqrt{\big(1-\alpha^2\big)\big(1-\beta^2\big)}(u_1+\alpha\beta u_3),\qquad w_2=\sqrt{\big(1-\alpha^2\big)\big(1-\beta^2\big)}(u_1-\alpha\beta u_3)
\end{gather*}
and
\begin{gather*}
Z_1=\mbox{sn}(w_1,\widetilde{\tau}_1)\mbox{sn}(w_2,\widetilde{\tau}_2),\qquad Z_2=\mbox{cn}(w_1,\widetilde{\tau}_1)\mbox{cn}(w_2,\widetilde{\tau}_2),\qquad Z_3=\mbox{dn}(w_1,\widetilde{\tau}_1)\mbox{dn}(w_2,\widetilde{\tau}_2).
\end{gather*}
Let $\wp_{j,k}$ be the hyperelliptic functions associated with the curve~$V$ defined by (\ref{4.7.1}).

\begin{Theorem}[{\cite[p.~3448]{BE}, \cite[p.~99]{BEL-2012}, \cite{Enolskii-Salerno-1996}}]\label{2021.11.20.AAA}
For ${\bf u}={}^t(u_1,u_3)\in\mathbb{C}^2$, the following relations hold:
\begin{gather}
Z_1=-\frac{\big(1-\alpha^2\big)\big(1-\beta^2\big)\big(\alpha^2\beta^2+\wp_{1,3}({\bf u})\big)}{\big(\alpha^2+\beta^2\big)\big(\wp_{1,3}({\bf u})-\alpha^2\beta^2\big)+\alpha^2\beta^2\wp_{1,1}({\bf u})+\wp_{3,3}({\bf u})},\label{4.29.111}\\
Z_2=-\frac{\big(1+\alpha^2\beta^2\big)\big(\alpha^2\beta^2-\wp_{1,3}({\bf u})\big)-\alpha^2\beta^2\wp_{1,1}({\bf u})-\wp_{3,3}({\bf u})}{\big(\alpha^2+\beta^2\big)\big(\wp_{1,3}({\bf u})-\alpha^2\beta^2\big)+\alpha^2\beta^2\wp_{1,1}({\bf u})+\wp_{3,3}({\bf u})},\label{4.29.112}\\
Z_3=-\frac{\alpha^2\beta^2\wp_{1,1}({\bf u})-\wp_{3,3}({\bf u})}{\big(\alpha^2+\beta^2\big)\big(\wp_{1,3}({\bf u})-\alpha^2\beta^2\big)+\alpha^2\beta^2\wp_{1,1}({\bf u})+\wp_{3,3}({\bf u})}. \label{4.29.113}
\end{gather}
\end{Theorem}

\begin{Theorem}[{\cite[p.~3448]{BE}, \cite[p.~100]{BEL-2012}, \cite{Enolskii-Salerno-1996}}]\label{4.29.999}
For ${\bf u}={}^t(u_1,u_3)\in\mathbb{C}^2$, the following relations hold:
\begin{gather*}
\wp_{1,1}({\bf u})=\frac{\big(\alpha^2+\beta^2\big)(Z_2-Z_3)+\big(1+\alpha^2\beta^2\big)(Z_3-1)}{Z_1+Z_2-1},
\\
\wp_{1,3}({\bf u})=\alpha^2\beta^2\frac{1+Z_1-Z_2}{Z_1+Z_2-1},
\\
\wp_{3,3}({\bf u})=\alpha^2\beta^2\frac{\big(\alpha^2+\beta^2\big)(Z_2+Z_3)-\big(1+\alpha^2\beta^2\big)(Z_3+1)}{Z_1+Z_2-1}.
\end{gather*}
\end{Theorem}

\begin{Remark}
In~\cite[p.~3448]{BE}, \cite[p.~366]{BE2}, \cite[pp.~79,~80]{BEL}, \cite[pp.~99,~175]{BEL-2012}, and \cite{Enolskii-Salerno-1996}, it is mentioned that $(Z_1,Z_2,Z_3)$ and $(\wp_{1,1},\wp_{1,3},\wp_{3,3})$ are coordinates on the Kummer surfaces.
In~\cite{BCMS2021}, algebraic correspondences between Kummer surfaces associated with the Jacobian variety of a~hyperelliptic curve of genus 2 admitting a morphism of degree~2 to an elliptic curve and
Kummer surfaces associated with the product of two non-isogenous elliptic curves are considered.
In \cite[equations~(2.56) and~(2.57)]{BCMS2021}, the formulae in Theorems~\ref{2021.11.20.AAA} and~\ref{4.29.999} in the present paper are described in terms of coordinates on the Kummer surfaces.
\end{Remark}

In this section, we give relationships between the above theorems (Theorems~\ref{2021.11.20.AAA} and~\ref{4.29.999}) and our results.
For $i=1,2$, let $\wp_{\widetilde{E}_i}$ be the Weierstrass elliptic function associated with $\widetilde{E}_i$.
Then, for $u\in\mathbb{C}$ and $i=1,2$, we have (cf.~\cite{H})
\begin{gather*}
\wp_{\widetilde{E}_i}(u)-\frac{1}{\kappa_i^2}=\frac{\mbox{cn}(u/\kappa_i, \widetilde{\tau}_i)^2}{\kappa_i^2\mbox{sn}(u/\kappa_i, \widetilde{\tau}_i)^2},
\\
\wp_{\widetilde{E}_i}(u)-1=\frac{\mbox{dn}(u/\kappa_i, \widetilde{\tau}_i)^2}{\kappa_i^2\mbox{sn}(u/\kappa_i, \widetilde{\tau}_i)^2},
\\
\wp_{\widetilde{E}_i}(u)=\frac{1}{\kappa_i^2\mbox{sn}(u/\kappa_i, \widetilde{\tau}_i)^2}.
\end{gather*}
From the above equations, for $u\in\mathbb{C}$ and $i=1,2$, we have
\begin{gather}
\mbox{sn}(u, \widetilde{\tau}_i)^2=\frac{1}{\kappa_i^2\wp_{\widetilde{E}_i}(\kappa_iu)},\label{4.29.114}
\\
\mbox{cn}(u, \widetilde{\tau}_i)^2 =\frac{\wp_{\widetilde{E}_i}(\kappa_iu)-1/\kappa_i^2}{\wp_{\widetilde{E}_i}(\kappa_iu)},
\label{4.29.115}
\\
\mbox{dn}(u, \widetilde{\tau}_i)^2 =\frac{\wp_{\widetilde{E}_i}(\kappa_iu)-1}{\wp_{\widetilde{E}_i}(\kappa_iu)}. \label{4.29.116}
\end{gather}

For $i=1,2$, our elliptic curves $E_i$ are isomorphic to the elliptic curves $\widetilde{E}_i$ by the following morphisms
\begin{gather}
\widetilde{\xi}_1\colon\quad E_1\to\widetilde{E}_1,\qquad (X,Y)\mapsto\big(\widetilde{X},\widetilde{Y}\big)=\bigg({-}\frac{(\alpha\beta-1)^2}{(\alpha-\beta)^2}X+1, \frac{\sqrt{-1}(\alpha\beta-1)^3Y}{(\alpha-\beta)^3}\bigg),\label{4.28.12345}\\
\widetilde{\xi}_2\colon\quad E_2\to\widetilde{E}_2,\qquad (X,Y)\mapsto\big(\widetilde{X},\widetilde{Y}\big)=\bigg({-}\frac{(\alpha\beta+1)^2}{(\alpha+\beta)^2}X+1, \frac{-\sqrt{-1}(\alpha\beta+1)^3Y}{(\alpha+\beta)^3}\bigg).\nonumber
\end{gather}

\begin{Proposition}\label{4.29.117}
For $u\in\mathbb{C}$, the following relations hold:
\begin{gather}
\wp_{E_1}\bigg({-}\sqrt{-1} \frac{\alpha\beta-1}{\alpha-\beta}u\bigg) =\frac{(\alpha-\beta)^2}{(\alpha\beta-1)^2}\big(1-\wp_{\widetilde{E}_1}(u)\big),\label{4.28.b}\\
\wp_{E_2}\bigg(\sqrt{-1} \frac{\alpha\beta+1}{\alpha+\beta}u\bigg) =\frac{(\alpha+\beta)^2}{(\alpha\beta+1)^2}\big(1-\wp_{\widetilde{E}_2}(u)\big). \label{4.28.c}
\end{gather}
\end{Proposition}

\begin{proof}We take a point $u\in\mathbb{C}\backslash\widetilde{\Lambda}_1$.
Then, there exist a~point $\widetilde{S}=\big(\widetilde{X},\widetilde{Y}\big)\in \widetilde{E}_1\backslash\{\infty\}$ and a~path~$\widetilde{\gamma}$ from $\infty$ to $\widetilde{S}$ on $\widetilde{E}_1$ such that
\begin{gather}
u=\int_{\infty}^{\widetilde{S}}\bigg({-}\frac{{\rm d}\widetilde{X}}{2\widetilde{Y}}\bigg),\label{4.28.1234}
\end{gather}
where we take $\widetilde{\gamma}$ as the path of integration in the right hand side of~(\ref{4.28.1234}).
Let $\gamma$ be the image of $\widetilde{\gamma}$ with respect to the map $\widetilde{\xi}_1^{-1}\colon \widetilde{E}_1\to E_1$, which is the inverse mapping of the map $\widetilde{\xi}_1$.
Let $S=\widetilde{\xi}_1^{-1}\big(\widetilde{S}\big)$ and $S=(X,Y)$.
The pullback of the holomorphic one form $-{\rm d}\widetilde{X}/\big(2\widetilde{Y}\big)$ on $\widetilde{E}_1$ with respect to the map $\widetilde{\xi}_1$ is
\begin{gather*}
\sqrt{-1} \frac{\alpha-\beta}{\alpha\beta-1}\bigg({-}\frac{{\rm d}X}{2Y}\bigg).
\end{gather*}
Then we have
\begin{gather}
u=\sqrt{-1} \frac{\alpha-\beta}{\alpha\beta-1}\int_{\infty}^S\bigg({-}\frac{{\rm d}X}{2Y}\bigg),\label{4.28.123}
\end{gather}
where we take $\gamma$ as the path of integration in the right hand side of (\ref{4.28.123}).
From (\ref{4.28.1234}), we have
\begin{gather}
\widetilde{X}=\wp_{\widetilde{E}_1}(u).\label{6.6.1}
\end{gather}
From (\ref{4.28.123}), we have
\begin{gather}
X=\wp_{E_1}\bigg({-}\sqrt{-1} \frac{\alpha\beta-1}{\alpha-\beta}u\bigg).\label{4.28.a}
\end{gather}
From (\ref{4.28.12345}), we have
\begin{gather}
X=\frac{(\alpha-\beta)^2}{(\alpha\beta-1)^2}\big(1-\widetilde{X}\big).\label{6.6.2}
\end{gather}
From (\ref{6.6.1}), (\ref{4.28.a}), and (\ref{6.6.2}), we find that the relation (\ref{4.28.b}) holds for $u\in\mathbb{C}\backslash\widetilde{\Lambda}_1$.
Since $\mathbb{C}\backslash\widetilde{\Lambda}_1$ is an open set in $\mathbb{C}$, the relation (\ref{4.28.b}) holds for $u\in\mathbb{C}$.
By changing $\beta$ with $-\beta$ in (\ref{4.28.b}), we obtain (\ref{4.28.c}).
\end{proof}

{\sloppy
By raising the both sides of (\ref{4.29.111}), (\ref{4.29.112}), and (\ref{4.29.113}) to the second power and substitu\-ting~(\ref{4.29.114}), (\ref{4.29.115}), and (\ref{4.29.116}), we can express
$\wp_{\widetilde{E}_1}(\kappa_1w_1)$ and $\wp_{\widetilde{E}_2}(\kappa_2w_2)$ in terms of $\wp_{1,1}({\bf u})$, $\wp_{1,3}({\bf u})$, and~$\wp_{3,3}({\bf u})$. We~have
\begin{gather*}
\kappa_1w_1=\sqrt{-1}(\alpha-\beta)(u_1+\alpha\beta u_3),\qquad
\kappa_2w_2=\sqrt{-1}(\alpha+\beta)(u_1-\alpha\beta u_3).
\end{gather*}}
From Proposition \ref{4.29.117} and the fact that the Weierstrass elliptic functions $\wp_{E_1}$ and $\wp_{E_2}$ are the even functions, we have
\begin{gather*}
\wp_{E_1}\big((1-\alpha\beta)(u_1+\alpha\beta u_3)\big) =\frac{(\alpha-\beta)^2}{(\alpha\beta-1)^2}\big(1-\wp_{\widetilde{E}_1}(\kappa_1w_1)\big),
\\
\wp_{E_2}\big((\alpha\beta+1)(u_1-\alpha\beta u_3)\big) =\frac{(\alpha+\beta)^2}{(\alpha\beta+1)^2}\big(1-\wp_{\widetilde{E}_2}(\kappa_2w_2)\big).
\end{gather*}
Therefore, we can express $\wp_{E_1}\big((1-\alpha\beta)(u_1+\alpha\beta u_3)\big)$ and $\wp_{E_2}\big((\alpha\beta+1)(u_1-\alpha\beta u_3)\big)$ in terms of~$\wp_{1,1}({\bf u})$, $\wp_{1,3}({\bf u})$, and $\wp_{3,3}({\bf u})$.
These expressions are equal to (\ref{4.13.112}) and (\ref{4.13.333}). We~set ${\bf u}={\bf k}_1v$. Then we have $w_1=2h$ and $w_2=0$, where
\begin{gather*}
h=\frac{\sqrt{\big(1-\alpha^2\big)\big(1-\beta^2\big)}}{1-\alpha\beta}v.
\end{gather*}
Since $\mbox{sn}(0,\widetilde{\tau}_2)=0$ and $\mbox{cn}(0,\widetilde{\tau}_2)=\mbox{dn}(0,\widetilde{\tau}_2)=1$, we have
\begin{gather}
Z_1=0, \qquad
Z_2=\mbox{cn}(2h,\widetilde{\tau}_1),\qquad
Z_3=\mbox{dn}(2h,\widetilde{\tau}_1).
\label{4.29.222}
\end{gather}
By the addition formulae for the Jacobi elliptic functions, we have
\begin{gather}
\mbox{cn}(2h,\widetilde{\tau}_1) =\frac{\mbox{cn}(h,\widetilde{\tau}_1)^2-\mbox{sn}(h,\widetilde{\tau}_1)^2 \mbox{dn}(h,\widetilde{\tau}_1)^2}{1-\kappa_1^2\mbox{sn}(h,\widetilde{\tau}_1)^4},\label{4.29.333}
\\
\mbox{dn}(2h,\widetilde{\tau}_1) =\frac{\mbox{dn}(h,\widetilde{\tau}_1)^2-\kappa_1^2\mbox{sn}(h,\widetilde{\tau}_1)^2 \mbox{cn}(h,\widetilde{\tau}_1)^2}{1-\kappa_1^2\mbox{sn}(h,\widetilde{\tau}_1)^4}.\label{4.29.444}
\end{gather}
From Theorem \ref{4.29.999}, (\ref{4.29.222}), (\ref{4.29.333}), (\ref{4.29.444}), (\ref{4.29.114}), (\ref{4.29.115}), and (\ref{4.29.116}), we obtain
\begin{gather*}
\wp_{1,1}({\bf k}_1v)=\frac{\alpha^2\beta^2-2\alpha\beta\wp_{\widetilde{E}_1} (\kappa_1h)+1}{1-\wp_{\widetilde{E}_1}(\kappa_1h)},
\\
\wp_{1,3}({\bf k}_1v)=-\alpha^2\beta^2,
\\
\wp_{3,3}({\bf k}_1v)=\alpha^2\beta^2\big\{\alpha^2+\beta^2-(\alpha-\beta)^2 \wp_{\widetilde{E}_1}(\kappa_1h)\big\}.
\end{gather*}
From (\ref{4.28.b}), we obtain (\ref{3.23.1}) and (\ref{3.23.2}).
In the same way as $\wp_{j,k}({\bf k}_1v)$, we can obtain (\ref{3.24.1}) and~(\ref{3.24.2}) from Theorem \ref{4.29.999}.

\begin{Remark}
For $i=1,2$ and $j=1,2,3$, let ${\rm al}^{(i)}_j(u)$ be the Weierstrass elliptic al functions associated with $\widetilde{E}_i$, where the definitions of the Weierstrass elliptic al functions are written in Section~\ref{2021.11.20.1111}.
From (\ref{4.29.114}), (\ref{4.29.115}), and (\ref{4.29.116}), for $i=1,2$, we obtain
\begin{gather*}
{\rm sn}(w_i,\widetilde{\tau}_i)=\frac{1}{\kappa_i {\rm al}^{(i)}_3(\kappa_iw_i)},\qquad
{\rm cn}(w_i,\widetilde{\tau}_i)=\frac{{\rm al}^{(i)}_1(\kappa_iw_i)}{{\rm al}^{(i)}_3(\kappa_iw_i)},\qquad
{\rm dn}(w_i,\widetilde{\tau}_i)=\frac{{\rm al}^{(i)}_2(\kappa_iw_i)}{{\rm al}^{(i)}_3(\kappa_iw_i)}.
\end{gather*}
Let
\begin{gather*}
\mathcal{Z}_j={\rm al}^{(1)}_j(\kappa_1w_1){\rm al}^{(2)}_j(\kappa_2w_2),\qquad j=1,2,3.
\end{gather*}
From Theorem \ref{2021.11.20.AAA}, for ${\bf u}={}^t(u_1,u_3)\in\mathbb{C}^2$, the following relations hold:
\begin{gather*}
\mathcal{Z}_1=\frac{\alpha^2-\beta^2}{\big(1-\alpha^2\big)^2\big(1-\beta^2\big)^2} \frac{\big(1+\alpha^2\beta^2\big)\big(\wp_{1,3}({\bf u})-\alpha^2\beta^2\big)+\alpha^2\beta^2\wp_{1,1}({\bf u})+\wp_{3,3}({\bf u})}{\alpha^2\beta^2+\wp_{1,3}({\bf u})},
\\
\mathcal{Z}_2=\frac{\alpha^2-\beta^2}{\big(1-\alpha^2\big)^2\big(1-\beta^2\big)^2} \frac{\wp_{3,3}({\bf u})-\alpha^2\beta^2\wp_{1,1}({\bf u})}{\alpha^2\beta^2+\wp_{1,3}({\bf u})},
\\
\mathcal{Z}_3=\frac{\alpha^2-\beta^2}{\big(1-\alpha^2\big)^2\big(1-\beta^2\big)^2} \frac{\big(\alpha^2+\beta^2\big)\big(\wp_{1,3}({\bf u})-\alpha^2\beta^2\big)+\alpha^2\beta^2\wp_{1,1}({\bf u})+\wp_{3,3}({\bf u})}{\alpha^2\beta^2+\wp_{1,3}({\bf u})}.
\end{gather*}
Furthermore, from Theorem \ref{4.29.999}, for ${\bf u}={}^t(u_1,u_3)\in\mathbb{C}^2$, the following relations hold:
\begin{gather*}
\wp_{1,1}({\bf u})=\big(1-\alpha^2\big)\big(1-\beta^2\big)\frac{\big(1+\alpha^2\beta^2\big)\mathcal{Z}_3 -\big(1-\alpha^2\big)\big(1-\beta^2\big)\mathcal{Z}_2-\big(\alpha^2+\beta^2\big)\mathcal{Z}_1}{\big(1-\alpha^2\big)\big(1-\beta^2\big) (\mathcal{Z}_3-\mathcal{Z}_1)+\alpha^2-\beta^2},
\\
\wp_{1,3}({\bf u})=\alpha^2\beta^2\frac{\big(1-\alpha^2\big)\big(1-\beta^2\big)(\mathcal{Z}_1-\mathcal{Z}_3) +\alpha^2-\beta^2}{\big(1-\alpha^2\big)\big(1-\beta^2\big)(\mathcal{Z}_3-\mathcal{Z}_1)+\alpha^2-\beta^2},
\\
\wp_{3,3}({\bf u})=\alpha^2\beta^2\big(1-\alpha^2\big)\big(1-\beta^2\big)\frac{\big(1+\alpha^2\beta^2\big)\mathcal{Z}_3 +\big(1-\alpha^2\big)\big(1-\beta^2\big)\mathcal{Z}_2-\big(\alpha^2+\beta^2\big)\mathcal{Z}_1}{\big(1-\alpha^2\big)\big(1-\beta^2\big) (\mathcal{Z}_3-\mathcal{Z}_1)+\alpha^2-\beta^2}.
\end{gather*}
\end{Remark}

\appendix

\section{Proof of Theorem \ref{3.16}}

For the details of the theory of algebraic function fields, see \cite{St}.
Let $H$ be a hyperelliptic curve of genus 2, $W$ be an elliptic curve, and $\rho\colon H\to W$ be a morphism of degree 2.
Let $\mathcal{F}(H)$ and $\mathcal{F}(W)$ be the function fields of $H$ and $W$, respectively. The map $\rho$ induces the
injective homomorphism of fields $\rho^*\colon \mathcal{F}(W)\to\mathcal{F}(H)$.
Since the degree of the map $\rho$ is 2, the degree of the extension $[\mathcal{F}(H): \mathcal{F}(W)]$ is 2.
Therefore the extension $\mathcal{F}(W)\subset\mathcal{F}(H)$ is the Galois extension.
Let $g$ be the generator of the Galois group $\mbox{Gal}(\mathcal{F}(H)/\mathcal{F}(W))$.
The field $\mathcal{F}(H)$ has the unique subfield $\mathbb{C}(x)$ of genus 0 such that $[\mathcal{F}(H): \mathbb{C}(x)]=2$, where $x\in\mathcal{F}(H)$ is transcendental over $\mathbb{C}$ (cf.~\cite[Proposition~6.2.4]{St}).
Since the genus of $g(\mathbb{C}(x))$ is 0, we have $g(\mathbb{C}(x))=\mathbb{C}(x)$.
Therefore the map $g$ induces the automorphism of $\mathbb{C}(x)$ fixing $\mathbb{C}$.

\begin{Lemma}
We can choose $X\in\mathbb{C}(x)$ such that $\mathbb{C}(X)=\mathbb{C}(x)$ and $g(X)=-X$.
\end{Lemma}

\begin{proof}
There exist $a,b,c,d\in\mathbb{C}$ such that $ad-bc=1$ and
\begin{gather*}
g(x)=\frac{ax+b}{cx+d}
\end{gather*}
(cf.~\cite[Theorem~1.1.1]{Sc}).
Let $A=\left(\begin{smallmatrix}a&b\\c&d\end{smallmatrix}\right)$.
Since $g^2$ is the identity mapping, we have $A^2=E$ or $A^2=-E$, where $E$ is the identity matrix.
By the Cayley--Hamilton theorem, we have
\begin{gather}
A^2=(a+d)A-E.\label{5.3.1}
\end{gather}
We assume $A^2=E$.
From (\ref{5.3.1}), we have $(a+d)A=2E$.
Thus we have $a+d\neq0$ and $A=\frac{2}{a+d}E$.
Therefore we have $g(x)=x$. Then we have $x\in\mathcal{F}(W)$, i.e., $\mathbb{C}(x)\subset\mathcal{F}(W)$.
Since the genus of $\mathbb{C}(x)$ is 0 and the genus of $\mathcal{F}(W)$ is 1, we have $\mathbb{C}(x)\neq\mathcal{F}(W)$.
Since $[\mathcal{F}(H): \mathcal{F}(W)]=2$, we have $[\mathcal{F}(H): \mathbb{C}(x)]>2$.
This contradicts $[\mathcal{F}(H): \mathbb{C}(x)]=2$.
Therefore we have $A^2=-E$. From (\ref{5.3.1}), we have $a+d=0$.
From $a+d=0$ and $ad-bc=1$, we find that the eigenvalues of $A$ are $\pm\sqrt{-1}$.
Therefore there exists a regular matrix $R$ such that
\begin{gather}
RA=\begin{pmatrix}-\sqrt{-1}&0\\0&\sqrt{-1}\end{pmatrix}R.\label{5.4.1}
\end{gather}
Let $R=\left(\begin{smallmatrix}r_1&r_2\\r_3&r_4\end{smallmatrix}\right)$ and
\begin{gather*}
X=\frac{r_1x+r_2}{r_3x+r_4}.
\end{gather*}
Then we have $\mathbb{C}(X)=\mathbb{C}(x)$ (cf.~\cite[Example~1.1.3]{Sc}).
Let $R'=RA$ and $R'=\left(\begin{smallmatrix}r_1'&r_2'\\r_3'&r_4'\end{smallmatrix}\right)$.
From~(\ref{5.4.1}), we have
\begin{gather*}
g(X)=\frac{r_1'x+r_2'}{r_3'x+r_4'}=\frac{-\sqrt{-1}r_1x-\sqrt{-1}r_2}{\sqrt{-1}r_3x+\sqrt{-1}r_4}=-X. \tag*{\qed}
\end{gather*}
\renewcommand{\qed}{}
\end{proof}

Since $\mathcal{F}(H)$ is an algebraic function field of genus 2, there exists $Y\in\mathcal{F}(H)$ such that $\mathcal{F}(H)=\mathbb{C}(X,Y)$ and
$Y^2=f(X)$ with a square-free polynomial $f(X)$ of degree 5 or 6 (cf.~\cite[Example~3.7.6]{St}). We~assume that the degree of $f(X)$ is 5.
Let $f(X)=(X-\gamma_1)\cdots(X-\gamma_5)$, where $\gamma_i\in\mathbb{C}$ for $1\le i\le5$ and $\gamma_i\neq\gamma_j$ for $i\neq j$.
A Weierstrass point of~$\mathcal{F}(H)$ is a place of $\mathcal{F}(H)$ which includes $X-\gamma_i$ for some $1\le i\le 5$ or $1/X$ (cf.~\cite[Corollary~I.49]{NG}).
Let $P_i$ be the place which includes $X-\gamma_i$.
Then $g(P_i)$ includes $X+\gamma_i$.
Since $g(P_i)$ is also a Weierstrass point of $\mathcal{F}(H)$ (cf.~\cite[Proposition~I.51]{NG}) and $g(P_i)$ does not include $1/X$, $g(P_i)$ includes $X-\gamma_j$ for some $1\le j\le 5$.
Therefore we have $-\gamma_i=\gamma_j$.
Thus the polynomial $f(X)$ has the following form:
\begin{gather*}
f(X)=X\big(X^2-\widetilde{\gamma}_1^2\big)\big(X^2-\widetilde{\gamma}_2^2\big),
\end{gather*}
where $\widetilde{\gamma}_1, \widetilde{\gamma}_2\in\mathbb{C}$, $\widetilde{\gamma}_1\widetilde{\gamma}_2\neq0$, and $\widetilde{\gamma}_1^2\neq\widetilde{\gamma}_2^2$.
From $Y^2=f(X)$, we have $g\big(Y^2\big)=-Y^2$. Thus we have $g(Y)=\sqrt{-1} Y$ or $g(Y)=-\sqrt{-1} Y$.
Then we have $g^2(Y)=-Y$, which contradicts that $g^2$ is the identity mapping.
Therefore the degree of $f(X)$ is 6.
Let $f(X)=(X-\gamma_1)\cdots(X-\gamma_6)$, where $\gamma_i\in\mathbb{C}$ for $1\le i\le6$ and $\gamma_i\neq\gamma_j$ for $i\neq j$.
A Weierstrass point of $\mathcal{F}(H)$ is a place of~$\mathcal{F}(H)$ which includes $X-\gamma_i$ for some $1\le i\le 6$ (cf.~\cite[Corollary~I.49]{NG}).
Let $P_i$ be the place which includes $X-\gamma_i$.
Then $g(P_i)$ includes $X+\gamma_i$.
Since $g(P_i)$ is also a Weierstrass point of~$\mathcal{F}(H)$ (cf.~\cite[Proposition~I.51]{NG}), $g(P_i)$ includes $X-\gamma_j$ for some $1\le j\le 6$.
Therefore we have $-\gamma_i=\gamma_j$.
Thus the polynomial $f(X)$ has the following form:
\begin{gather*}
f(X)=\big(X^2-\widetilde{\gamma}_1^2\big)\big(X^2-\widetilde{\gamma}_2^2\big)\big(X^2-\widetilde{\gamma}_3^2\big),
\end{gather*}
where $\widetilde{\gamma}_1, \widetilde{\gamma}_2, \widetilde{\gamma}_3\in\mathbb{C}$, $\widetilde{\gamma}_1\widetilde{\gamma}_2\widetilde{\gamma}_3\neq0$, and $\widetilde{\gamma}_i^2\neq\widetilde{\gamma}_j^2$ for $i\neq j$.
Let $\widetilde{H}$ be the hyperelliptic curve of genus 2 defined by
\begin{gather*}
\widetilde{H}=\big\{(X,Y)\in\mathbb{C}^2\mid Y^2=f(X)\big\}.
\end{gather*}
Since the function field of $\widetilde{H}$ is $\mathcal{F}(H)$, the curve $H$ is birationally equivalent to the curve $\widetilde{H}$.
Since the curves $H$ and $\widetilde{H}$ are smooth, the curve $H$ is isomorphic to the curve $\widetilde{H}$.
The curve $\widetilde{H}$ is isomorphic to the curve
\begin{gather*}
H'=\big\{(s,t)\in\mathbb{C}^2\mid t^2=\big(s^2-1\big)\big(s^2-e_1^2\big)\big(s^2-e_2^2\big)\big\},
\end{gather*}
where $e_1=\widetilde{\gamma}_2/\widetilde{\gamma}_1$ and $e_2=\widetilde{\gamma}_3/\widetilde{\gamma}_1$, by the morphism
\begin{gather*}
\widetilde{\xi}\colon\quad \widetilde{H}\to H',\qquad (X,Y)\mapsto\big(X/\widetilde{\gamma}_1, Y/\widetilde{\gamma}_1^3\big).
\end{gather*}
Thus the curve $H$ is isomorphic to the curve $H'$. Next, we consider the curve $H'$ defined by~(\ref{3.12.2}).
Let $W$ be the elliptic curve defined by
\begin{gather*}
W=\big\{(s,t)\in\mathbb{C}^2\mid t^2=(s-1)\big(s-e_1^2\big)\big(s-e_2^2\big)\big\}.
\end{gather*}
Then we can define the morphism of degree 2
\begin{gather*}
H'\to W\colon\quad (s,t)\mapsto\big(s^2,t\big).
\end{gather*}
Therefore we obtain the statement of the theorem.

\subsection*{Acknowledgements}

The authors would like to thank the referees for reading our
manuscript carefully and giving the useful comments. The work of Takanori Ayano was supported by JSPS KAKENHI Grant Number JP21K03296 and was partly supported by Osaka City University Advanced Mathema\-ti\-cal Institute (MEXT Joint Usage/Research Center on Mathematics and Theoretical Physics JPMXP0619217849).

\pdfbookmark[1]{References}{ref}
\LastPageEnding


\begin{thebibliography}{99}
\footnotesize\itemsep=0pt

\bibitem{A-E-E-2004}
Athorne C., Eilbeck J.C., Enolskii V.Z., A {${\rm SL}(2)$} covariant theory of
 genus 2 hyperelliptic functions, \href{https://doi.org/10.1017/S030500410300728X}{\textit{Math. Proc. Cambridge Philos. Soc.}}
 \textbf{136} (2004), 269--286.

\bibitem{AB2019}
Ayano T., Buchstaber V.M., Ultraelliptic integrals and two-dimensional sigma
 functions, \href{https://doi.org/10.1134/S0016266319030018}{\textit{Funct. Anal. Appl.}} \textbf{53} (2019), 157--173.

\bibitem{AB2020}
Ayano T., Buchstaber V.M., Analytical and number-theoretical properties of the
 two-dimensional sigma function, \href{https://doi.org/10.22405/2226-8383-2020-21-1-9-50}{\textit{Chebyshevski\u{\i} Sb.}} \textbf{21}
 (2020), 9--50, \href{https://arxiv.org/abs/2003.08565}{arXiv:2003.08565}.

\bibitem{Baker-1907}
Baker H.F., An introduction to the theory of multiply periodic functions,
 Cambridge University Press, Cambridge, 1907.

\bibitem{BE0}
Belokolos E.D., Bobenko A.I., Enol'skii V.Z., Its A.R., Matveev V.B.,
 Algebro-geometric approach to nonlinear integrable equations, \textit{Springer Series
 in Nonlinear Dynamics}, Springer-Verlag, Berlin, 1994.

\bibitem{BE}
Belokolos E.D., Enolskii V.Z., Reduction of abelian functions and algebraically
 integrable systems.~{I}, \href{https://doi.org/10.1023/A:1011983313249}{\textit{J.~Math. Sci.}} \textbf{106} (2001),
 3395--3486.

\bibitem{BE2}
Belokolos E.D., Enolskii V.Z., Reduction of abelian functions and algebraically
 integrable systems.~{II}, \href{https://doi.org/10.1023/A:1012800600273}{\textit{J.~Math. Sci.}} \textbf{108} (2002),
 295--374.

\bibitem{B-H-2004}
Birkenhake C., Lange H., Complex abelian varieties, 2nd ed., \textit{Grundlehren der
 mathematischen Wissenschaften}, Vol.~302, \href{https://doi.org/10.1007/978-3-662-06307-1}{Springer-Verlag}, Berlin,
 2004.

\bibitem{B-W-2003}
Birkenhake C., Wilhelm H., Humbert surfaces and the {K}ummer plane,
 \href{https://doi.org/10.1090/S0002-9947-03-03238-0}{\textit{Trans. Amer. Math. Soc.}} \textbf{355} (2003), 1819--1841.

\bibitem{O-Bolza-1887}
Bolza O., Ueber die Reduction hyperelliptischer Integrale erster Ordnung und
 erster Gattung auf elliptische durch eine Transformation vierten Grades,
 \href{https://doi.org/10.1007/BF02440001}{\textit{Math. Ann.}} \textbf{28} (1887), 447--456.

\bibitem{O-Bolza-1935}
Bolza O., Der singul\"are {F}all der {R}eduktion hyperelliptischer {I}ntegrale
 erster {O}rdnung auf elliptische durch eine {T}ransformation dritten
 {G}rades, \href{https://doi.org/10.1007/BF01472234}{\textit{Math. Ann.}} \textbf{111} (1935), 477--500.

\bibitem{BCMS2021}
Braeger N., Clingher A., Malmendier A., Spatig S., Isogenies of certain {K}3
 surfaces of rank~18, \href{https://doi.org/10.1007/s40687-021-00293-0}{\textit{Res. Math. Sci.}} \textbf{8} (2021), 57,
 60~pages, \href{https://arxiv.org/abs/2109.03189}{arXiv:2109.03189}.

\bibitem{B-L-S-2013}
Br\"oker R., Lauter K., Streng M., Abelian surfaces admitting an
 {$(l,l)$}-endomorphism, \href{https://doi.org/10.1016/j.jalgebra.2013.07.011}{\textit{J.~Algebra}} \textbf{394} (2013), 374--396,
 \href{https://arxiv.org/abs/1106.1884}{arXiv:1106.1884}.

\bibitem{BD}
Bruin N., Doerksen K., The arithmetic of genus two curves with {$(4,4)$}-split
 {J}acobians, \href{https://doi.org/10.4153/CJM-2011-039-3}{\textit{Canad.~J. Math.}} \textbf{63} (2011), 992--1024,
 \href{https://arxiv.org/abs/0902.3480}{arXiv:0902.3480}.

\bibitem{BEL-97-1}
Buchstaber V.M., Enolskii V.Z., Leykin D.V., Hyperelliptic {K}leinian functions
 and applications, in Solitons, Geometry, and Topology: on the Crossroad,
 \textit{Amer. Math. Soc. Transl. Ser.~2}, Vol.~179, \href{https://doi.org/10.1090/trans2/179/01}{Amer. Math. Soc.},
 Providence, RI, 1997, 1--33, \href{https://arxiv.org/abs/solv-int/9603005}{arXiv:solv-int/9603005}.

\bibitem{BEL-97-2}
Buchstaber V.M., Enolskii V.Z., Leykin D.V., Kleinian functions, hyperelliptic
 {J}acobians and applications, \textit{Rev. Math. Math. Phys.} \textbf{10} (1997),
 3--120.

\bibitem{BEL}
Buchstaber V.M., Enolskii V.Z., Leykin D.V., Hyperelliptic abelian functions,
 1997, {a}vailable at \url{https://www.researchgate.net/publication/266955336_Kleinian_functions_hyperelliptic_Jacobians_and_applications}.

\bibitem{BEL-99-R}
Buchstaber V.M., Enolskii V.Z., Leykin D.V., Rational analogues of abelian
 functions, \href{https://doi.org/10.1007/BF02465189}{\textit{Funct. Anal. Appl.}} \textbf{33} (1999), 83--94.

\bibitem{BEL-2012}
Buchstaber V.M., Enolskii V.Z., Leykin D.V., Multi-dimensional sigma-functions,
 \href{https://arxiv.org/abs/1208.0990}{arXiv:1208.0990}.

\bibitem{BL-2004}
Buchstaber V.M., Leykin D.V., Heat equations in a nonholonomic frame,
 \href{https://doi.org/10.1023/B:FAIA.0000034039.92913.8a}{\textit{Funct. Anal. Appl.}} \textbf{38} (2004), 88--101.

\bibitem{Cassels}
Cassels J.W.S., Flynn E.V., Prolegomena to a middlebrow arithmetic of curves of
 genus {$2$}, \textit{London Mathematical Society Lecture Note Series}, Vol.~230, \href{https://doi.org/10.1017/CBO9780511526084}{Cambridge University Press}, Cambridge, 1996.

\bibitem{FCoppini2020}
Coppini F., Grinevich P.G., Santini P.M., Effect of a small loss or gain in the
 periodic nonlinear {S}chr\"odinger anomalous wave dynamics, \href{https://doi.org/10.1103/physreve.101.032204}{\textit{Phys.
 Rev.~E}} \textbf{101} (2020), 032204, 8~pages, \href{https://arxiv.org/abs/1910.13176}{arXiv:1910.13176}.

\bibitem{E-H-K-K-L-1}
Enolski V.Z., Hackmann E., Kagramanova V., Kunz J., L\"ammerzahl C., Inversion
 of hyperelliptic integrals of arbitrary genus with application to particle
 motion in general relativity, \href{https://doi.org/10.1016/j.geomphys.2011.01.001}{\textit{J.~Geom. Phys.}} \textbf{61} (2011),
 899--921, \href{https://arxiv.org/abs/1011.6459}{arXiv:1011.6459}.

\bibitem{E-H-K-K-L-P-1}
Enolski V.Z., Hartmann B., Kagramanova V., Kunz J., L\"ammerzahl C., Sirimachan
 P., Inversion of a~general hyperelliptic integral and particle motion in
 {H}o\v{r}ava--{L}ifshitz black hole space-times, \href{https://doi.org/10.1063/1.3677831}{\textit{J.~Math. Phys.}}
 \textbf{53} (2012), 012504, 35~pages, \href{https://arxiv.org/abs/1106.2408}{arXiv:1106.2408}.

\bibitem{Enolskii-Salerno-1996}
Enolskii V.Z., Salerno M., Lax representation for two-particle dynamics
 splitting on two tori, \href{https://doi.org/10.1088/0305-4470/29/17/002}{\textit{J.~Phys.~A: Math. Gen.}} \textbf{29} (1996),
 L425--L431, \href{https://arxiv.org/abs/solv-int/9603004}{arXiv:solv-int/9603004}.

\bibitem{Fay-1973}
Fay J.D., Theta functions on {R}iemann surfaces, \textit{Lecture Notes in
 Math.}, Vol.~352, \href{https://doi.org/10.1007/BFb0060090}{Springer-Verlag}, Berlin~-- New York, 1973.

\bibitem{F2}
Flynn E.V., Coverings of curves of genus~2, in Algorithmic Number Theory
 ({L}eiden, 2000), \textit{Lecture Notes in Comput. Sci.}, Vol.~1838, Editor
 W.~Bosma, \href{https://doi.org/10.1007/10722028_3}{Springer}, Berlin, 2000, 65--84.

\bibitem{FK}
Frey G., Kani E., Curves of genus {$2$} covering elliptic curves and an
 arithmetical application, in Arithmetic Algebraic Geometry ({T}exel, 1989),
 \textit{Progr. Math.}, Vol.~89, Editors G.~van~der Geer, F.~Oort,
 J.~Steenbrink, \href{https://doi.org/10.1007/978-1-4612-0457-2_7}{Birkh\"auser Boston}, Boston, MA, 1991, 153--176.

\bibitem{NG}
G\"ob N., Automorphism groups of hyperelliptic function fields, Ph.D.~Thesis,
 {T}echnische Universit\"at Kaiserslautern, 2004.

\bibitem{G}
Grant D., Formal groups in genus two, \href{https://doi.org/10.1515/crll.1990.411.96}{\textit{J.~Reine Angew. Math.}}
 \textbf{411} (1990), 96--121.

\bibitem{Howe-Leprevost-Poonen-Large-torsion-subgroups-2000}
Howe E.W., Lepr\'evost F., Poonen B., Large torsion subgroups of split
 {J}acobians of curves of genus two or three, \href{https://doi.org/10.1515/form.2000.008}{\textit{Forum Math.}} \textbf{12}
 (2000), 315--364, \href{https://arxiv.org/abs/math.NT/9809210}{arXiv:math.NT/9809210}.

\bibitem{Hudson-1990}
Hudson R.W.H.T., Kummer's quartic surface, \textit{Cambridge Mathematical Library},
 Cambridge University Press, Cambridge, 1990.

\bibitem{Humbert-1899}
Humbert G., Sur les fonctions ab\'eliennes singuli\`eres (premier M\'emoire),
 \textit{J.~Math. Pures Appl.} \textbf{5} (1899), 233--350.

\bibitem{HumbertG-1900}
Humbert G., Sur les fonctions ab\'eliennes singuli\`eres (deuxi\`eme
 M\'emoire), \textit{J.~Math. Pures Appl.} \textbf{6} (1900), 279--386.

\bibitem{HumbertG-1901}
Humbert G., Sur les fonctions ab\'eliennes singuli\`eres (Troisi\`eme
 M\'emoire), \textit{J.~Math. Pures Appl.} \textbf{7} (1901), 97--124.

\bibitem{H}
Hurwitz A., Vorlesungen \"uber allgemeine {F}unktionentheorie und elliptische
 {F}unktionen, \textit{Die Grundlehren der mathematischen Wissenschaften},
 \href{https://doi.org/10.1007/978-3-662-00750-1}{Springer-Verlag}, Berlin~-- New York, 1964.

\bibitem{J}
Jacobi C.G.J., Review of Legendre, Th\'eorie des fonctions elliptiques,
 troisi\`eme suppl\'ement, \href{https://doi.org/10.1515/crll.1832.8.413}{\textit{J.~Reine Angew. Math.}} \textbf{1832}
 (1832), 413--417.

\bibitem{K}
Krazer A., Lehrbuch der Thetafunktionen, Leipzig, B.G.~Teubner, 1903.

\bibitem{Ku}
Kuhn R.M., Curves of genus {$2$} with split {J}acobian, \href{https://doi.org/10.2307/2000749}{\textit{Trans. Amer.
 Math. Soc.}} \textbf{307} (1988), 41--49.

\bibitem{Milne}
Milne J.S., Abelian Varieties, Version~2.0 (2008), available at
 \url{https://www.jmilne.org/math/CourseNotes/av.html}.

\bibitem{Mumford-1983}
Mumford D., Tata lectures on theta.~{II}, \textit{Modern Birkh\"auser Classics},
 \href{https://doi.org/10.1007/978-0-8176-4578-6}{Birkh\"auser Boston, Inc.}, Boston, MA, 2007.

\bibitem{N-2010}
Nakayashiki A., On algebraic expressions of sigma functions for {$(n,s)$}
 curves, \href{https://doi.org/10.4310/AJM.2010.v14.n2.a2}{\textit{Asian~J. Math.}} \textbf{14} (2010), 175--211,
 \href{https://arxiv.org/abs/0803.2083}{arXiv:0803.2083}.

\bibitem{P}
Platonov V.P., Number-theoretic properties of hyperelliptic fields and the
 torsion problem in {J}acobians of hyperelliptic curves over the rational
 number field, \href{https://doi.org/10.1070/RM2014v069n01ABEH004877}{\textit{Russian Math. Surveys}} \textbf{69} (2014), 1--34.

\bibitem{Pre}
Previato E., Victor {E}nolski (1945--2019), \href{https://doi.org/10.1090/noti}{\textit{Notices Amer. Math. Soc.}}
 \textbf{67} (2020), 1755--1767.

\bibitem{Sc}
Scognamillo R., Zannier U., Introductory notes on valuation rings and function
 fields in one variable, \textit{Appunti. Scuola Normale Superiore di Pisa
 (Nuova Serie)}, Vol.~14, \href{https://doi.org/10.1007/978-88-7642-501-1}{Edizioni della Normale}, Pisa, 2014.

\bibitem{J.P.Serre-2020}
Serre J.P., Rational points on curves over finite fields, \textit{Documents
 Math\'ematiques (Paris)}, Vol.~18, Soci\'et\'e Math\'ematique de France,
 Paris, 2020.

\bibitem{S}
Shaska T., Curves of genus 2 with {$(N,N)$} decomposable {J}acobians,
 \href{https://doi.org/10.1006/jsco.2001.0439}{\textit{J.~Symbolic Comput.}} \textbf{31} (2001), 603--617,
 \href{https://arxiv.org/abs/math.AG/0312285}{arXiv:math.AG/0312285}.

\bibitem{SV}
Shaska T., V\"olklein H., Elliptic subfields and automorphisms of genus~2
 function fields, in Algebra, Arithmetic and Geometry with Applications
 ({W}est {L}afayette, {IN}, 2000), \href{https://doi.org/10.1007/978-3-642-18487-1_42}{Springer}, Berlin, 2004, 703--723,
 \href{https://arxiv.org/abs/math.AG/0107142}{arXiv:math.AG/0107142}.

\bibitem{silverman}
Silverman J.H., The arithmetic of elliptic curves, 2nd ed., \textit{Graduate Texts in
 Mathematics}, Vol.~106, \href{https://doi.org/10.1007/978-0-387-09494-6}{Springer}, Dordrecht, 2009.

\bibitem{AOSmirnov}
Smirnov A.O., Periodic two-phase ``rogue waves'', \href{https://doi.org/10.1134/S0001434613110266}{\textit{Math. Notes}}
 \textbf{94} (2013), 897--907.

\bibitem{St}
Stichtenoth H., Algebraic function fields and codes, 2nd ed., \textit{Graduate Texts in
 Mathematics}, Vol.~254, \href{https://doi.org/10.1007/978-3-540-76878-4}{Springer-Verlag}, Berlin, 2009.

\bibitem{K.Weierstrass-1894}
Weierstrass K., Mathematische {W}erke~{I}, Mayer und M\"uller, Berlin, 1894.

\bibitem{W}
Wetherell J.L., Bounding the number of rational points on certain curves of
 high rank, Ph.D.~Thesis, {U}niversity of California, Berkeley, 1997.

\bibitem{WW}
Whittaker E.T., Watson G.N., A course of modern analysis, 4th~ed., \textit{Cambridge
 Mathematical Library}, \href{https://doi.org/10.1017/CBO9780511608759}{Cambridge University Press}, Cambridge, 1996.

\end{thebibliography}
\end{document}